\documentclass[a4paper,11pt]{article}
\usepackage[pdftex]{graphicx}
\usepackage{graphicx,a4wide,listings}
\usepackage[intlimits]{amsmath}
\usepackage{amsxtra, amssymb,fancyhdr, amsthm,latexsym}
\usepackage[english]{babel}
\usepackage[ansinew]{inputenc}
\usepackage[bookmarks]{}
\usepackage{enumerate}
\numberwithin{equation}{section}
\numberwithin{figure}{section}
 \usepackage[colorlinks=true]{hyperref}
\hypersetup{urlcolor=blue, citecolor=red}
\usepackage{enumitem}
\usepackage{tikz}
\usepackage{caption}
\usepackage{subfigure}

\begin{document}

\newtheorem{theorem}{Theorem}[section]
\newtheorem{corollary}[theorem]{Corollary}
\newtheorem{lemma}[theorem]{Lemma}
\newtheorem{proposition}[theorem]{Proposition}
\newtheorem{conjecture}[theorem]{Conjecture}
\newtheorem*{problem}{Problem}
\theoremstyle{definition}
\newtheorem{definition}[theorem]{Definition}
\theoremstyle{definition}
\newtheorem{remark}[theorem]{Remark}
\newcommand{\ep}{\varepsilon}
\newcommand{\eps}[1]{{#1}_{\varepsilon}}

\def\R{\mathbb{R}}
\def\Rp{\mathbb{R}^+}
\def\N{\mathbb{N}}
\def\Np{\mathbb{N}^+}
\def\calO{\mathcal{O}}
\def\calR{\mathcal{R}_{\lambda,\mu}}
\def\calRtilde{\tilde{\mathcal{R}}_{\lambda,\mu}}
\def\tF{\mathcal{F}^0}
\def\Fa{\mathcal{F}^{\alpha}}
\def\Fan{\mathcal{F}^{\alpha_n}}
\def\Fbak{\mathcal{F}^{\overline{\alpha}_k}}
\def\Fbakl{\mathcal{F}^{\overline{\alpha}_{k_\ell}}}
\def\vbak{v^{\overline{\alpha}_k}}
\def\vbakl{v^{\overline{\alpha}_{k_\ell}}}
\newcommand{\Int}[4]{\int_{#1}^{#2}\! #3 \, #4}
\newcommand{\pair}[2]{\left\langle #1 , #2 \right\rangle}
\newcommand{\map}[3]{#1 : #2 \to #3}
\newcommand{\sgn}{\operatorname*{sgn}}
\newcommand{\pd}[2]{\dfrac{\partial#1}{\partial#2}}

\title{Anisotropic interactions in a first-order aggregation model: a proof of concept}
\author{Joep H.M.~Evers\thanks{Corresponding author; Institute for Complex Molecular Systems \& Centre for Analysis, Scientific computing and Applications, Eindhoven University of Technology, P.O.~Box 513, 5600 MB Eindhoven, The Netherlands, email: \texttt{j.h.m.evers@tue.nl}.} \  \  Razvan C.~Fetecau\thanks{Department of Mathematics, Simon Fraser University, 8888 University Dr., Burnaby BC V5A 1S6, Canada, email: \texttt{van@sfu.ca}.} \  \  Lenya Ryzhik\thanks{Department of Mathematics, Stanford University, Stanford CA 94305, USA, email: \texttt{ryzhik@math.stanford.edu}.}}
\date{\today}
\maketitle

\begin{abstract}
We extend a well-studied ODE model for collective behaviour by considering anisotropic interactions among individuals. Anisotropy is modelled by limited sensorial perception of individuals, that depends on their current direction of motion. Consequently, the first-order model becomes  {\em implicit}, and new analytical issues, such as non-uniqueness and jump discontinuities in velocities, are being raised. We study the well-posedness of the anisotropic model and discuss its modes of breakdown. To extend solutions beyond breakdown we propose a relaxation system containing a small parameter $\ep$, which can be interpreted as a small amount of inertia or response time. We show that the limit $\ep\to0$ can be used as a jump criterion to select the physically correct velocities. In smooth regimes, the convergence of the relaxation system as $\ep\to0$ is guaranteed by a theorem due to Tikhonov. We illustrate the results with numerical simulations in two dimensions.
\end{abstract}


\section{Introduction}

Mathematical models for collective behaviour in biological aggregations have attracted a large interest in recent years. One extensively studied model describes the evolution of positions~$x_i$ ($i=1,\dots,N$) of $N$ particles (individuals) in $\mathbb{R}^d$:
\begin{subequations} \label{eqn:fo-nobz} \begin{align}
    \frac{dx_i}{dt} &= v_i, \\
    v_i &= -\frac{1}{N} \displaystyle\sum_{j\neq i}\nabla_{x_i}K(|x_i-x_j|). \label{eqn:vi-nobz} \end{align} \end{subequations}
%
Here, $K$ is an aggregation potential, which incorporates inter-individual social interactions such as long-range attraction and short-range repulsion. Its  form depends on the particular application at hand.

Due, in part, to the complexity of the nonlinear dynamics in \eqref{eqn:fo-nobz}, the theoretical research has focused to a large extent on its continuum limit, the  evolution equation for the aggregation density $\rho(x,t)$:
\begin{subequations} \label{eq:aggre}
\begin{align}
\rho_{t} &+\nabla\cdot(\rho v)=0,  \label{eqn:rhot} \\
v &=-\nabla K\ast\rho.
\end{align}
\end{subequations}
Here, $\ast$ denotes the spatial convolution. The density $\rho(x,t)$ represents a continuum approximation to the distribution of individuals in \eqref{eqn:fo-nobz} as $N\rightarrow\infty$ (a formal derivation of this fact can be found in \cite{BodnarVelasquez1}).

Both the discrete and the continuous models appear in various works on the mathematical models for biological aggregations -- see~\cite{M&K,TBL} for an extensive review, both of the literature and the relevance of the models. The same models also arise in a number of other applications, such as the granular media~\cite{CaMcVi2006}, the self-assembly of nanoparticles~\cite{ HoPu2006}, the Ginzburg-Landau vortices~\cite{DuZhang03} and the molecular dynamics simulations of matter~\cite{Haile1992}.  At the PDE level, the well-posedness of~\eqref{eq:aggre} was studied in~\cite{BodnarVelasquez2, BertozziLaurent, BertozziLaurentRosado}, and its long-time behaviour in in~\cite{Burger:DiFrancesco,LeToBe2009,FeHuKo11,FeHu13}. One focus of the analytical investigations was on the possibility of the blow-up of the solutions via mass concentration into one or several Dirac distributions when the potential $K$ is attractive~\cite{FeRa10,BertozziCarilloLaurent, HuBe2010}.

Numerical simulations of both models are almost exclusively based on the discretizations of the particle model \eqref{eqn:fo-nobz}. They have demonstrated a   wide variety of possible behaviours of  solutions  \cite{Brecht_etal2011,KoSuUmBe2011,LeToBe2009,BrechtUminsky2012,Balague_etalARMA,FeHuKo11}. As the system \eqref{eqn:fo-nobz} represents a gradient flow with respect to the energy
\begin{equation}
\label{eqn:energy}
E(x_1,\dots,x_N) = \frac{1}{N} \sum_{i,j\neq i} K(|x_i-x_j|),
\end{equation}
its long-time dynamics can be characterized by the extrema of this interaction function. They can be very diverse: uniform densities in a ball, uniform densities on a co-dimension one manifold (ring in 2D, sphere in 3D), annuli, soccer balls, etc. Many of these patterns are observed experimentally in self-assembled biological aggregations \cite{Parish:Keshet, Ballerini_etal, Camazine_etal, Theraulaz:Deneubourg}, which gives practical ground and motivation to the studies of this model.

In the biological applications,  \eqref{eqn:fo-nobz} is used to model animal aggregations, such as insect swarms, fish schools, bird flocks, etc.~\cite{M&K}.  The interaction potential in \eqref{eqn:fo-nobz} is  {\em isotropic}, as it depends only on the pairwise distances between individuals.  This assumption is often unrealistic, as most species have a restricted zone of social perception, defined by the limitations of their field of vision or of other perception senses \cite{Seidl:Kaiser, Kunz:Hemelrijk2012}. However, despite the extensive literature on model \eqref{eqn:fo-nobz},  there has been no systematic study of its (more realistic) anisotropic extensions.  The primary goal of this paper is to fill this gap.   We note that anisotropy/non-symmetry of interactions was considered, both analytically and numerically, in works on {\em second-order} aggregation models \cite{CarrilloVecil2010,Frouvelle-anisotropic,AlbiPareschi2013}, where the velocity is governed by a differential equation itself. However, adding anisotropy to {\em first-order} models such  as \eqref{eqn:fo-bz}, though similar conceptually, is very different at a mathematical and numerical level.

We introduce perception restrictions  in   \eqref{eqn:fo-nobz} via weights in \eqref{eqn:vi-nobz}
that   limit  the influence by individuals $j$ on the reference individual $i$:
 \begin{equation}
 \label{eqn:v-weighted}
 v_i = -\frac{1}{N} \displaystyle\sum_{j\neq i}\nabla_{x_i}K(|x_i-x_j|) w_{ij}.
 \end{equation}
The choice of the weights $w_{ij}$ depends on what limitations on the field of perception one wants to consider. In this paper, we consider the social perception to be entirely {\em visual} and assume that individuals have a limited field of vision centred around their direction of motion. Given a reference individual located at $x_i$ moving with velocity $v_i$, the weights $w_{ij}$ should depend on the relative position $x_j-x_i$ of individual $j$ with respect to the current direction of motion $v_i$ of individual $i$ (such as whether individual $j$ is ahead or behind individual $i$). Mathematically, we model $w_{ij}$ as
 \begin{equation}
 \label{eqn:weights}
 w_{ij}= g\left(\dfrac{x_i-x_j}{|x_i-x_j|}\cdot \dfrac{v_i}{|v_i|}\right),
 \end{equation}
with a function $g$ chosen so that $w_{ij}$ are largest when $j$ is right ahead of individual $i$ ($x_j-x_i$ is in the same direction of $v_i$) and lowest when $j$ is right behind individual $i$ ($x_j-x_i$ in the opposite direction of $v_i$). Note that the weights $w_{ij}$ are {\em not} symmetric -- in general, $w_{ij} \neq w_{ji}$.

With \eqref{eqn:v-weighted} and \eqref{eqn:weights}, the original model \eqref{eqn:fo-nobz} becomes
\begin{subequations}
\label{eqn:fo-bz}
\begin{align}
    \frac{dx_i}{dt} &= v_i, \\
    v_i &= -\frac{1}{N} \displaystyle\sum_{j\neq i}\nabla_{x_i}K(|x_i-x_j|)\,g\left(\dfrac{x_i-x_j}{|x_i-x_j|} \cdot \dfrac{v_i}{|v_i|}\right), \label{eqn:vi-bz}
  \end{align}
\end{subequations}
the aggregation model   we study in this paper. The velocities $v_i$ are no longer explicitly given in terms of the spatial configuration $\{x_1,x_2,\dots,x_N\}$ as in \eqref{eqn:vi-nobz}, but are defined instead through the {\em implicit} equation \eqref{eqn:vi-bz}, which, in general, may  have multiple solutions. Hence, {\em non-uniqueness} of the velocity is a major issue immediately brought up by the anisotropic extension \eqref{eqn:fo-bz}.

The second important issue is the {\em loss of smoothness} of solutions of \eqref{eqn:fo-bz}.
The roots of \eqref{eqn:vi-bz} may disappear dynamically, as the spatial configuration $\{x_1, x_2, \dots, x_N\}$ changes in time. Hence, velocities have to be allowed to be discontinuous at these jump times, and a selection criteria for the allowable/physical jumps should be defined and enforced. Finally, a third issue is that velocities in \eqref{eqn:fo-bz} can become zero (particles can {\em stop}) in finite time, and the model, at least as it appears in \eqref{eqn:fo-bz}, is not even defined when some $v_i=0$.


The main tool in dealing with the issues above is to introduce a relaxation term in the equation for the velocities $v_i$. More precisely, we consider the following regularized system
\begin{subequations}
\label{eqn:so}
\begin{align}
    \frac{dx_i}{dt} &= v_i, \\
    \ep\dfrac{dv_i}{dt} &= -v_i -\dfrac1N\displaystyle\sum_{j\neq i}\nabla_{x_i}K(|x_i-x_j|)\,g\left(\dfrac{x_i-x_j}{|x_i-x_j|} \cdot \dfrac{v_i}{|v_i|} \right) .
      \end{align}
\end{subequations}
From the biological point of view, \eqref{eqn:so} introduces a small response time for the individuals.
Mechanically, \eqref{eqn:fo-nobz} was derived in \cite{BodnarVelasquez1}  from   Newton's law   by taking mass to be zero with $g \equiv 1$ (fully isotropic). We bring back the original second-order model, with a small mass $\ep$, as this
is essential in dealing with the anisotropic interactions.

The system \eqref{eqn:so} is well-posed, as solutions exist (locally) and are unique. Using
an old theorem of Tikhonov \cite{Tikhonov1952,Vasileva1963}, the limit $\ep \to 0$ in  \eqref{eqn:so} can be performed in the smooth regime,  provided that the solutions of \eqref{eqn:fo-bz} are asymptotically stable in a certain sense (see  Section \ref{sect:conv} for details). At the times of the velocity jumps in \eqref{eqn:fo-bz}, we use the relaxation model  \eqref{eqn:so} to enforce a  ``physical" jump selection criteria. The relaxation term in \eqref{eqn:so} is shown to smooth out the trajectories of  \eqref{eqn:fo-bz} -- see Figure \ref{fig:trajectories full and zoom}(b,c) for an illustration.
The modes of breakdown of \eqref{eqn:fo-bz} and the jump selection through the relaxation model \eqref{eqn:so} are presented in Section~\ref{sect:breakdown}.

The main goal of this paper is to demonstrate   how the regularization  \eqref{eqn:so} can be used as an analytical, and a numerical tool to understand and simulate solutions to  \eqref{eqn:fo-bz}.  We do not deal here with extensive numerical simulations and the complex issue of the long term behaviour of \eqref{eqn:fo-bz}. We restrict ourselves to the two dimensional simulations of \eqref{eqn:fo-bz}, where we show how \eqref{eqn:so} can be used to deal with instantaneous root losses, as well as particle stopping.

\begin{figure}%
\centering
\includegraphics[width=0.5\textwidth]{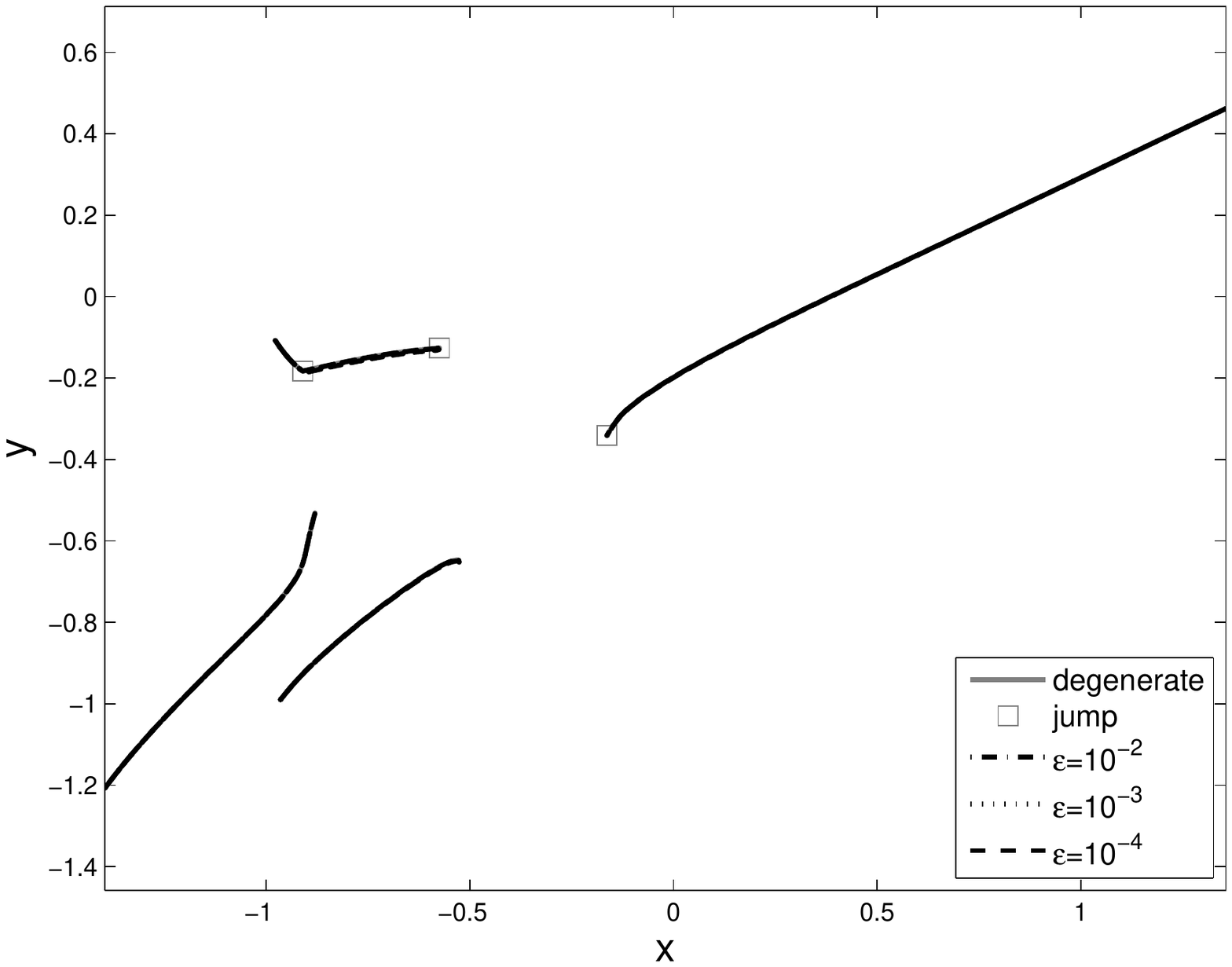}\\
(a)\vspace{0.1cm}\\ \vspace{0.5cm}
\includegraphics[width=0.45\textwidth]{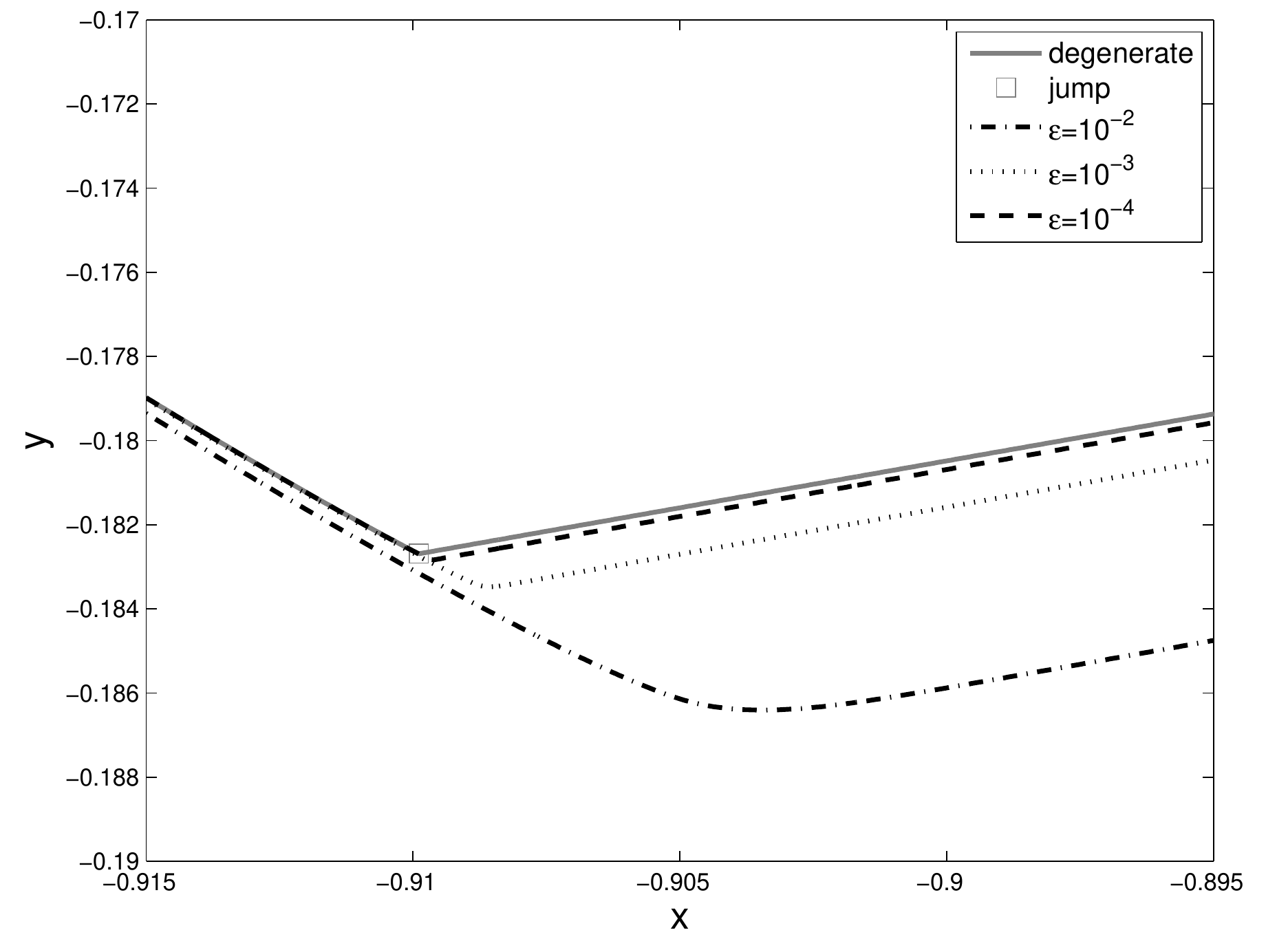}
$~~~$
\includegraphics[width=0.45\textwidth]{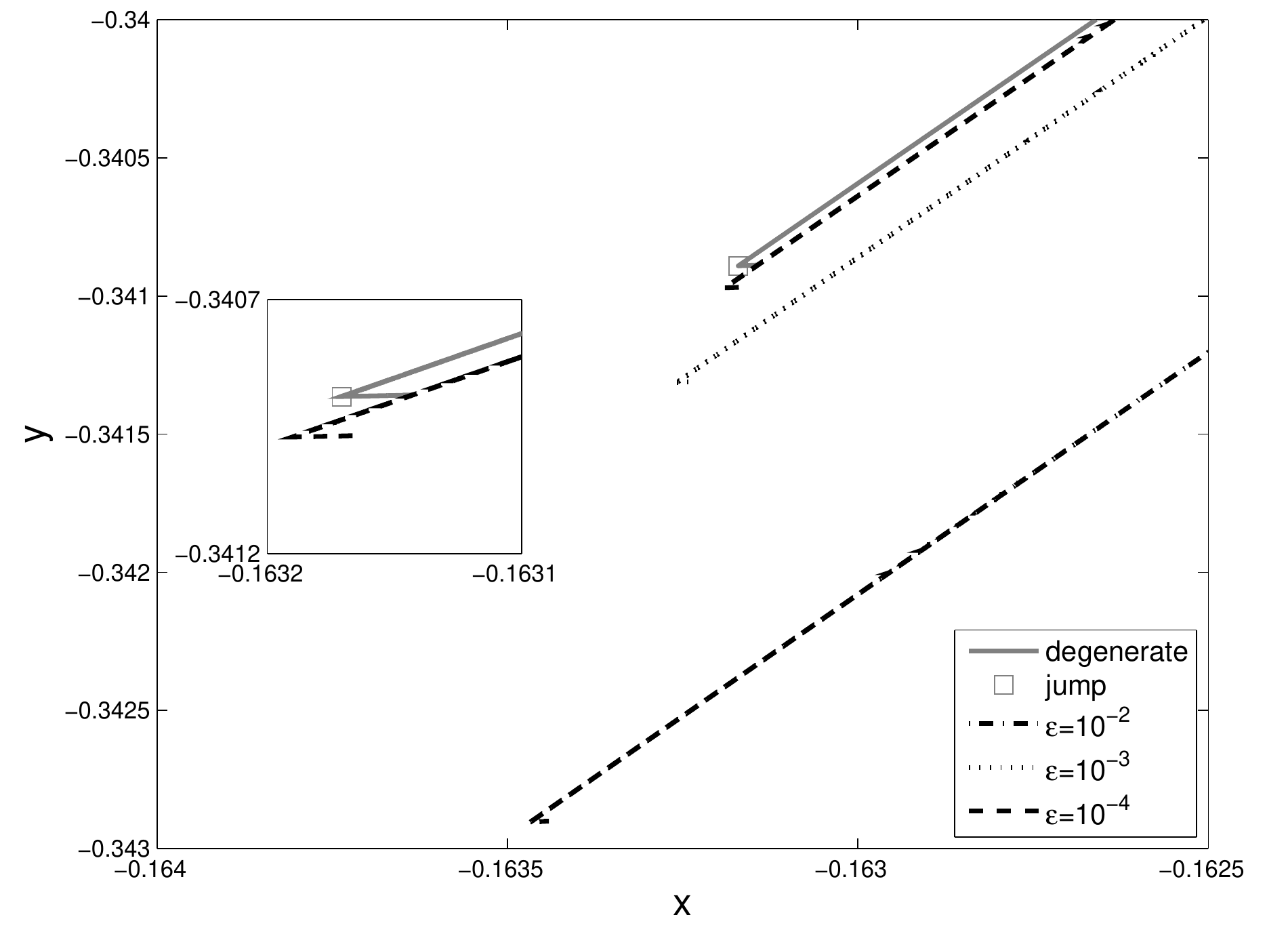}\\
(b) \hspace{0.45\textwidth} (c)
\caption{Time evolution of a random initial configuration of four particles. (a) The solid line represents the solution of the anisotropic first-order model \eqref{eqn:fo-bz}. The extension of the solution beyond breakdown times (indicated by squares) is explained in Section \ref{sect:breakdown}. On top of this plot we graph the solution of the relaxation model \eqref{eqn:so} for three values of $\ep$: $\ep=10^{-2},10^{-3}$, and $10^{-4}$. The plots are indistinguishable at the scale of the figure. (b) Zoomed images near two of the breakdown times of model \eqref{eqn:fo-bz}.  Note how the $\ep$-model \eqref{eqn:so} captures the discontinuities in velocity, as well as approximates solutions of \eqref{eqn:fo-bz} away from the jumps.
}%
\label{fig:trajectories full and zoom}%
\end{figure}


\section{The anisotropic model \eqref{eqn:fo-bz}}
\label{sect:fom}
First, given a fixed spatial configuration $\{x_i\}_{i=1}^{N}$, we study the existence of a velocity field that satisfies the fixed point equation \eqref{eqn:vi-bz}. Then we study the dynamic evolution of solutions to \eqref{eqn:fo-bz}, initialized at some configuration  $x_i(0)= x_i^0$.

\subsection{The interaction kernel}
\label{sect:prelim}
As discussed in the Introduction, the system \eqref{eqn:fo-bz} with $g\equiv1$ is a well-established model, extensively studied in the last decade. The properties of the interaction potential $K$ are crucial for the well-posedness and the long-time behaviour of the solutions to \eqref{eqn:fo-nobz} (or \eqref{eq:aggre}). We are interested in this work in biologically relevant choices of $K$ which incorporate short-range repulsive and long-range attractive interactions. One such choice is the Morse potential~\cite{LeToBe2009,Chuang_etal}, which has the form
\begin{equation}
\label{eqn:Morse}
K(|x|) = -C_a e^{-|x|/l_a} + C_r e^{-|x|/l_r},
\end{equation}
with the constants $C_a$, $C_r$ and $l_a$, $l_r$ representing the strengths and ranges of the
attractive and repulsive interactions, respectively. The theoretical results in this paper apply both to the Morse potential, and to a large set of other choices of K (for example, power-laws with positive exponents and the antiderivative of the $\tanh$ function \cite{KoSuUmBe2011,Brecht_etal2011}).

The function $g$ that models the field of vision is the main new ingredient in this paper, and its
choice is far from unique.
Denote by $\phi_{ij}$  the angle between $x_j-x_i$ and $v_i$ -- see Figure~\ref{fig:vision}(a):
 \[
\dfrac{x_i-x_j}{|x_i-x_j|} \cdot \dfrac{v_i}{|v_i|} = - \cos \phi_{ij}.
\]
The weights $w_{ij} = g(-\cos \phi_{ij})$ should be the largest ($=1$) for $\phi_{ij}=0$ (the vectors $x_j-x_i$ and $v_i$ are parallel) and
the lowest (possibly $0$) for~$\phi_{ij}=\pi$ ($x_j-x_i$ and $v_i$ anti-parallel). Here are two choices of $g$ that capture this behaviour:
\begin{equation}
\label{eqn:gtanh}
g(-\cos\phi)=[\tanh(a(\cos\phi + 1 - b/\pi))+1]/c,
\end{equation}
with $c$ a normalization constant such that $g(-1)=1$, and
\begin{equation}
\label{eqn:glin}
g(-\cos\phi)= [a \cos\phi + b]/ (a+b).
\end{equation}
The $\tanh$ function \eqref{eqn:gtanh} is illustrated in Figure \ref{fig:vision}(b).  The function takes values close to $1$ in the field of vision (around $\phi=0$) and decays steeply toward the blind zone. These regions of high values, steep descent and low values are indicated in dark grey, light grey and white in Figure \ref{fig:vision}(a). In \eqref{eqn:gtanh} the parameter $a$ controls the steepness of the graph and $b$ controls its width (size of field of vision).

\begin{figure}%
\centering
\includegraphics[totalheight=0.3\textheight]{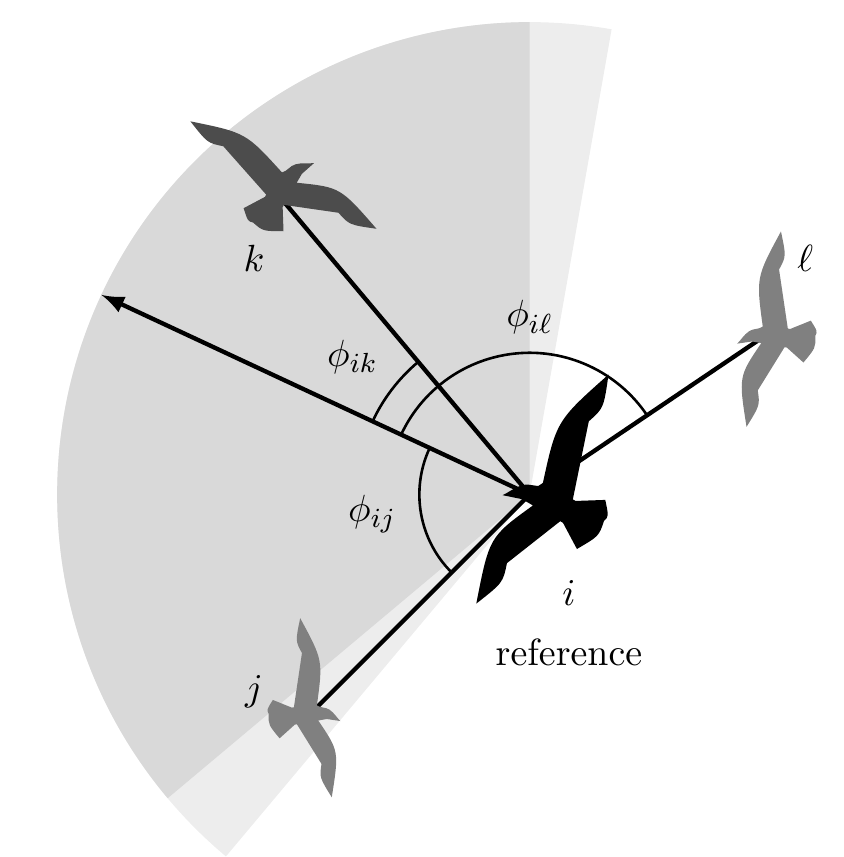}
$~~~$
\includegraphics[totalheight=0.25\textheight]{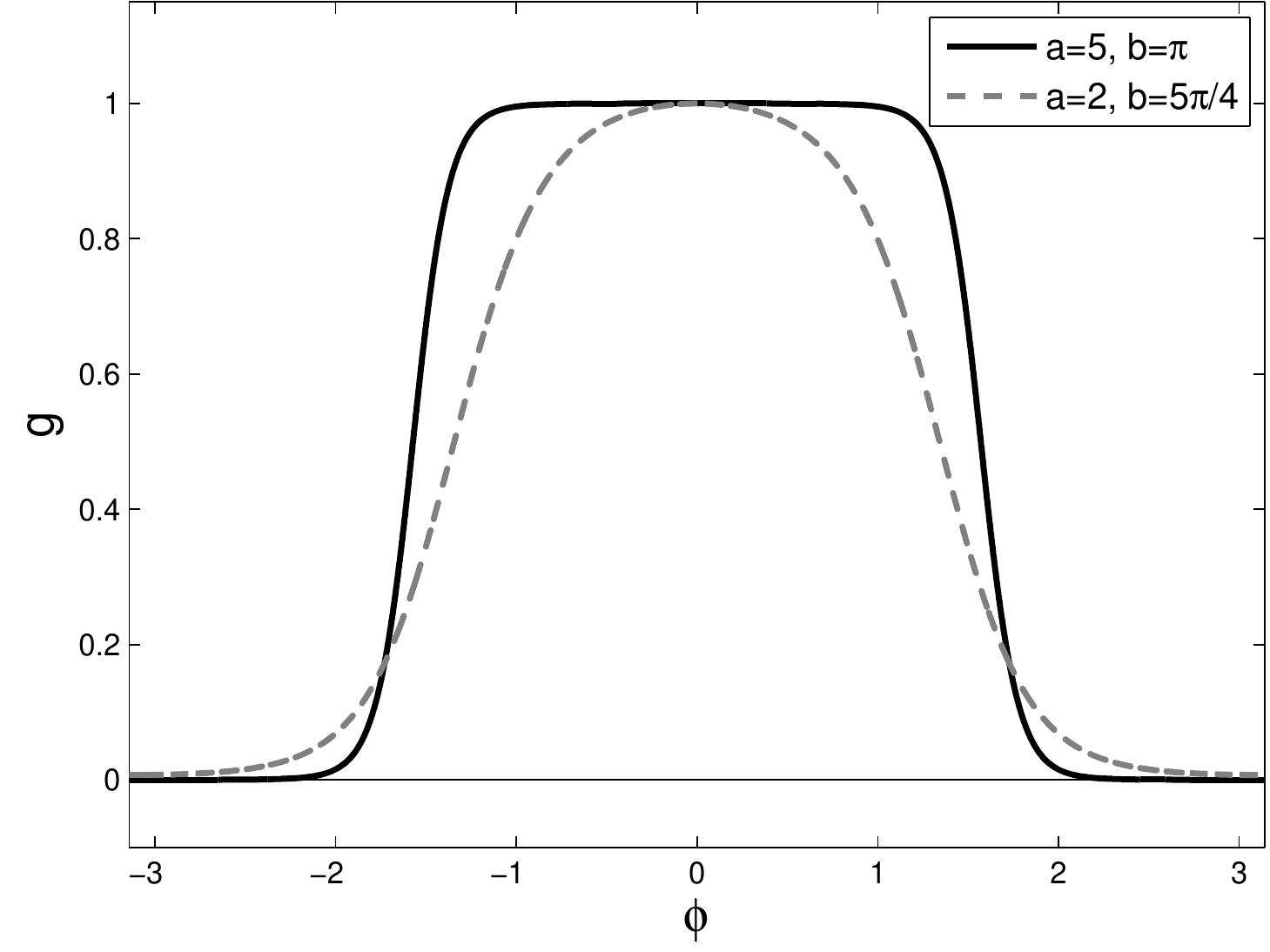}\\
 (a) \hspace{7.2cm} (b)
\caption{(a) An illustration of the visual perception of a reference individual $i$:
the field of vision (dark grey), the peripheral vision (light grey) and the blind zone (white).
         Interactions are weighted:
$w_{ik}>w_{ij}>w_{i\ell}$. (b) The weight function $g$ given by \eqref{eqn:gtanh}. The following parameters are shown: $a=5$ and $b=\pi$ (solid), $a=2$ and $b=5\pi/4$ (dashed)  --- $a$ controls the steepness of the graph and $b$ controls its width. The function takes values close to $1$ in a region around $\phi=0$ (field of vision), has a steep decay to nearly $0$ in the peripheral vision, and takes negligible values near $\phi = \pm \pi$ (blind zone).}%
\label{fig:vision}%
\end{figure}

As we are not concerned in this paper with sharp analytical results,  we will assume that $K$ and $g$ satisfy enough properties for the analysis to be carried over simply and with the least technical difficulties. Some of the results can be obtained under weaker assumptions than others and this fact will be pointed out when appropriate. In general, the following assumptions on $K$ and $g$ are needed:
\begin{equation}
\label{eqn:hypK}
 K: \mathbb{R}^+ \to \mathbb{R} \quad \textrm{ is $C^2$, with bounded derivatives},
\end{equation}
and
\begin{equation}
\label{eqn:hypg}
g:[-1,1] \to [0,1]  \quad \textrm{ is } C^1 \textrm{ with a bounded derivative}.
\end{equation}


\subsection{The implicit equation for $v_i$ -- existence and non-uniqueness}
\label{sect:fp}

Next, we investigate the existence and uniqueness of a fixed point of the implicit equation~\eqref{eqn:vi-bz} for $v_i$.
Note that particle stopping ($v_i=0$ for some $i$) is not well-defined for~\eqref{eqn:fo-bz}. The reason is that the field of vision of an individual is intrinsically defined in terms of its current direction of motion, along $v_i$. However, we observe in numerical simulations that particles do have
a tendency to stop. Stopping may occur for instance when a particle loses sense of the others, brakes down and stops before making a sudden turn to redirect itself toward the rest of the group. Or in an opposite situation, when a particle gets to a point where its repulsive interactions are dominant, and makes a turn to avoid getting too close to the rest. Such a sudden change in direction due to stopping is illustrated in Figure \eqref{fig:trajectories full and zoom}(c).

In order to deal with the stopping,
as is common in the ODE theory with discontinuous nonlinearities~\cite{Crandall-Liggett,Filippov1988},
we introduce a {\em generalized} definition of a fixed point of~\eqref{eqn:vi-bz}.
We will  regard $\sgn(z)$ as
a set-valued function given by the subdifferential of the Euclidean norm $|z|$:
\begin{equation}
  \sgn(z):=\partial |z| = \left\{
  \begin{array}{ll}
   \vspace{0.2cm}
    \frac{z}{|z|} & \quad \hbox{$z \neq0$,} \\
       \overline{B}(0,1) & \quad \hbox{$z=0$.}
  \end{array}
\right.
\end{equation}
Given a spatial configuration $\{x_i\}$, the resting scenario can now be considered as a solution of \eqref{eqn:vi-bz} if the following generalization of a solution is taken.
\begin{definition}[Generalized fixed point]\label{def generalization fixed point}
We call $v\in\R^d$ a \textit{generalized} solution of \eqref{eqn:vi-bz} if there exists an $s\in\sgn(v)$ such that
\begin{equation}
v = -\dfrac1N\displaystyle\sum_{j\neq i}\,\nabla_{x_i}K(|x_i-x_j|)\, g\left(\dfrac{(x_i-x_j)}{|x_i-x_j|}\cdot s \right).\label{generalized fixed point}
\end{equation}
\end{definition}
We show in Theorem \ref{theorem existence generalized fixed point} that the implicit equation \eqref{eqn:vi-bz} always has at least one generalized solution in the sense of Definition \ref{def generalization fixed point}. However,  as the next example shows, such solutions are
not expected to be unique.


\subsubsection*{Non-uniqueness}

In order to show that solutions of \eqref{eqn:vi-bz} are generally non-unique, we look at a simple example in two dimensions ($d=2$) with four particles ($N=4$) situated at the four corners of a square, where {\em each} equation for $v_i$ ($i=1,\dots,4$) has three solutions. Hence, there are $81=3^4$ different combinations of $v_i$ that solve \eqref{eqn:vi-bz} ($i=1,\dots,4$)  for this particular example!
To be more precise, we take the anisotropy function $g$ to be linear, as in~\eqref{eqn:glin}: $g(s)=(1-s)/2$, and $K$ to be the Morse potential \eqref{eqn:Morse}. For such $K$, the derivative $K'(r)$ is negative (repulsive) at short distances and positive (attractive) at long ranges. It is easy to see that we can find $\beta>0$ with  $K'(\beta)<0$ ($\beta$ in the repulsive range) such that
\begin{equation}
K'(\beta) + \frac12\sqrt{2}\,K'(\beta\,\sqrt{2}) = 0.\label{eqn:find zero w beta +}
\end{equation}
Let now the four particles be located in the corners of a square of size $\beta$ (see Figure \ref{fig:Config non-uniqueness}):
\begin{equation*}
x_1=\frac{\beta}{2}\,{1\brack1},\,\,x_2=\frac{\beta}{2}\,{-1\brack1},\,\,x_3=\frac{\beta}{2}\,{-1\brack-1},\,\,x_4=\frac{\beta}{2}\,{1\brack-1}.
\end{equation*}
Then,  for each particle there are three admissible velocities in the sense of Definition \ref{def generalization fixed point}, as illustrated in Figure \ref{fig:Config non-uniqueness}. The first is  the stopping/zero velocity indicated by a circle. The second is a velocity vector pointing inward (toward the centre) and the third solution is a velocity pointing outward,  opposite in direction and equal in size to the previous.
\begin{figure}[ht]
\centering
\begin{tikzpicture}[scale=0.7, >= latex]
	\draw[->, gray!25!white, line width = 0.5mm] (0,-4)--(0,4);
	\draw[->, gray!25!white, line width = 0.5mm] (-4,0)--(6,0);

    \draw[dashed] (-2,-2) rectangle (2,2);

    \draw [gray!80!white, line width = 0.5mm] (2,2) circle (0.12cm) node[anchor=west,black]{$\hspace{1 mm}x_1$};
    \draw [gray!80!white, line width = 0.5mm, ->] (2.12,2.12) -- (3,3);
    \draw [gray!80!white, line width = 0.5mm, ->] (1.88,1.88) -- (1,1);

    \draw [gray!80!white, line width = 0.5mm] (-2,2) circle (0.12cm) node[anchor=east,black]{$x_2\hspace{1 mm}$};
    \draw [gray!80!white, line width = 0.5mm, ->] (-2.12,2.12) -- (-3,3);
    \draw [gray!80!white, line width = 0.5mm, ->] (-1.88,1.88) -- (-1,1);

    \draw [gray!80!white, line width = 0.5mm] (-2,-2) circle (0.12cm) node[anchor=east,black]{$x_3\hspace{1 mm}$};
    \draw [gray!80!white, line width = 0.5mm, ->] (-2.12,-2.12) -- (-3,-3);
    \draw [gray!80!white, line width = 0.5mm, ->] (-1.88,-1.88) -- (-1,-1);

    \draw [gray!80!white, line width = 0.5mm] (2,-2) circle (0.12cm) node[anchor=west,black]{$\hspace{1 mm}x_4$};
    \draw [gray!80!white, line width = 0.5mm, ->] (2.12,-2.12) -- (3,-3);
    \draw [gray!80!white, line width = 0.5mm, ->] (1.88,-1.88) -- (1,-1);

    \draw[<->] (4,-2)--(4,2)node[midway, right]{$\beta$};
	\draw[<->] (-2,-3)--(2,-3)node[midway, below]{$\beta$};
\end{tikzpicture}
\caption{Four particles positioned on the corners of a square of size $\beta$. Each of them has three admissible velocities (generalized solutions of \eqref{eqn:vi-bz}): one pointing inward, one pointing outward and $v=0$ (indicated by a circle).}
\label{fig:Config non-uniqueness}
\end{figure}
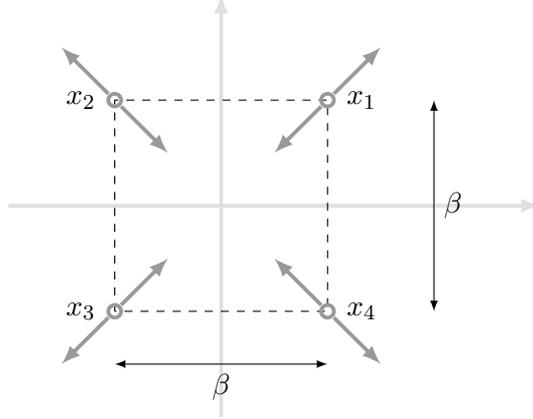
Indeed, consider, for instance, particle $1$ in this square configuration and the velocity equation \eqref{eqn:vi-bz} with $i=1$.
We have
\begin{equation}
x_1-x_2=\beta\,{1\brack0},\,\,x_1-x_3=\beta\,{1\brack1},\,\,x_1-x_4=\beta\,{0\brack1},\label{eqn:positions corners beta}
\end{equation}
and
\begin{equation}
|x_1-x_2|=\beta,\,\,|x_1-x_3|=\beta\,\sqrt{2},\,\,|x_1-x_4|=\beta,\label{eqn:norms corners beta}
\end{equation}
so that
\begin{align}
\label{tilde F counterexample}
-\dfrac14  & \sum_{j\neq 1}  \nabla_{x_1}K(|x_1-x_j|)\,g\left(\dfrac{(x_1-x_j)}{|x_1-x_j|}\cdot s \right) \nonumber \\
&= -\dfrac 18\underbrace{\sum_{j\neq 1}\, K'(|x_1-x_j|)\,\dfrac{x_1-x_j}{|x_1-x_j|}}_{=0} + \dfrac 18\sum_{j\neq 1}\, K'(|x_1-x_j|)\,\dfrac{x_1-x_j}{|x_1-x_j|} \,\, \dfrac{(x_1-x_j)}{|x_1-x_j|}\cdot s,
\end{align}
for any $s\in\bar{B}(0,1)$. The first term on the right-hand side vanishes due
to~\eqref{eqn:find zero w beta +}, \eqref{eqn:positions corners beta} and~\eqref{eqn:norms corners beta}\footnote{This is equivalent to the fact that the square configuration with size $\beta$ is an equilibrium of the isotropic model ($ g \equiv 1$)}.
For $v=0$, $s=0$ is an element of $\sgn(v)$ and
\begin{equation*}
0=\dfrac 18\sum_{j\neq 1}\, K'(|x_1-x_j|)\,\dfrac{x_1-x_j}{|x_1-x_j|} \,\, \dfrac{(x_1-x_j)}{|x_1-x_j|}\cdot 0.
\end{equation*}
Hence, $v_{1}=0$ is a {\em generalized} fixed point of \eqref{eqn:vi-bz}.

We look now for a non-zero solution $v=(v^1,v^2)$ of  \eqref{eqn:vi-bz}. Given \eqref{tilde F counterexample}, we have to solve for $v$ from
\begin{align*}
v &= \dfrac 18\sum_{j\neq 1}\, K'(|x_1-x_j|)\,\dfrac{x_1-x_j}{|x_1-x_j|} \,\, \dfrac{(x_1-x_j)}{|x_1-x_j|}\cdot \dfrac{v}{|v|} \nonumber \\
&= \dfrac{1}{8|v|}\,\left(v^1\,K'(\beta)\,{1\brack0} + \frac12\,(v^1+v^2)\,K'(\beta\,\sqrt{2})\,{1\brack1} + v^2\,K'(\beta)\,{0\brack1} \right).
\end{align*}
Looking for a particular solution with $v^1=v^2$, we find
\begin{equation*}
v^1 = \dfrac{1}{8\sqrt{2}}\,\dfrac{v^1}{|v^1|}\,\left(K'(\beta) + K'(\beta\,\sqrt{2})\right),
\end{equation*}
or, in view of \eqref{eqn:find zero w beta +},
\begin{equation}
\label{eqn:v1-nonuniq}
|v^1| = \dfrac{\sqrt{2}}{16}\,\left(1-\dfrac12\sqrt{2}\right)K'(\beta\,\sqrt{2}).
\end{equation}
Note that since $K'(\beta\,\sqrt{2})>0$, the right-hand side of
\eqref{eqn:v1-nonuniq} is, indeed, positive. Hence, there two
(opposite in sign, but equal in magnitude) solutions for $v^1$. This
yields two velocity vectors as illustrated in Figure \ref{fig:Config
  non-uniqueness}. The same argument applies to the other particles
due to the rotational symmetry.


\subsubsection*{Existence}

We now prove the existence of a generalized solution of \eqref{eqn:vi-bz}.
Let $\{x_j\}_{j=1}^N\subset\R^d$ be a {\em fixed} set of distinct
positions and take a specific index $i\in\{1,\ldots,N\}$.
\begin{theorem}
\label{theorem existence generalized fixed point}
Assume that $\map{K}{\Rp}{\R}$ has a bounded derivative, and
$\map{g}{[-1,1]}{[0,1]}$ is continuous.  Then there exists a
generalized fixed point in the sense of Definition \ref{def
  generalization fixed point}.
\end{theorem}
\begin{proof}
To deal with
the singularity of \eqref{eqn:vi-bz} at $v=0$, we   use a
regularization. For any $\alpha>0$, define the mapping
$\map{\Fa}{\R^d}{\R^d}$ by
\begin{equation}
\label{eqn:Falpha}
\Fa(v) := -\dfrac1N\sum_{j\neq i}\, \nabla_{x_i}K(|x_i-x_j|)\,\, g\left(\dfrac{(x_i-x_j)}{|x_i-x_j|}\cdot \dfrac{v}{\alpha+|v|} \right),~~\hbox{for all $v\in\R^d$.}
\end{equation}
This map is continuous and uniformly bounded on $\R^d$, with
\begin{equation*}
\label{bound on F}
|\Fa(v)| \leqslant \|K'\|_{\infty}\,\|g\|_{\infty}, \qquad \text {for all } v\in \mathbb{R}^d.
\end{equation*}
Brouwer's Fixed Point Theorem implies that $\Fa$ has a fixed point
(which depends on $\alpha$) in the closed ball ${{B(0,\rho)}}$
where~$\rho:=\|K'\|_{\infty}\,\|g\|_{\infty}$. We now show that a
generalized fixed point satisfying~\eqref{generalized fixed point} can
be obtained by passing to the limit $\alpha\downarrow0$.  Assume
that~$\alpha_n\to 0$, with~$\alpha_n>0$, and let~$\{v^{\alpha_n}\}$ be
a corresponding set of fixed points of $\Fan$:
\begin{equation}
\label{fixed point Fbak}
v^{\alpha_n} = -\dfrac1N\sum_{j\neq i}\, \nabla_{x_i}K(|x_i-x_j|)\,\, g\left(\dfrac{(x_i-x_j)}{|x_i-x_j|}\cdot \dfrac{v^{\alpha_n}}{\alpha_n+|v^{\alpha_n}|} \right).
\end{equation}
Since $|v^{\alpha_n}|\le\rho$ is uniformly bounded, $\{v^{\alpha_n}\}$
converges along a subsequence. For convenience, relabel this subsequence
as $\{v^{\alpha_{n}}\}$ and define its limit:
\[
v = \lim_{n\to\infty} v^{\alpha_n}.
\]
If $v\neq 0$, then, as
\[
\lim_{n\to \infty}
\frac{v^{\alpha_n}}{\alpha_n+|v^{\alpha_n}|} = \frac{v}{|v|},
\]
we can simply pass the limit $n \to \infty$ in
the fixed point equation \eqref{fixed point Fbak}, and conclude that $v$
is a fixed point.

On the other hand, if $v= 0$, we set
\begin{equation*}
w_n=\frac{v^{\alpha_n}}{\alpha_n+|v^{\alpha_n}|} ,
\end{equation*}
and note that $|w_n|<1$.
Thus, up to extraction of a subsequence, $w_n$ converges to a limit
\[
s = \lim_{k \to \infty} w_{n_k},
\]
with $|s|\le 1$.  Sending $k \to \infty$ in the fixed point equation
\eqref{fixed point Fbak} for $v^{\alpha_{n_k}}$ we find that $v=0$
satisfies \eqref{generalized fixed point}, with
$s\in\overline{B(0,1)}=\sgn(v)$.
\end{proof}


\subsection{Local continuity of trajectories}
\label{sec local continuity trajectories}
Given that a velocity field always exists for a given configuration,
we study now the local existence of continuous solutions to
\eqref{eqn:fo-bz}. We denote by bold characters $\bf{x}$ and $\bf{v}$
the concatenation of all particles' locations and velocities,
respectively, i.e.,
\[
\mathbf{x} = (x_1,\dots,x_N), \quad \mathbf{v} = (v_1,\dots,v_N).
\]
To rule out issues such as collisions or particle stopping, we look
for solutions $(\mathbf{x}(t),\mathbf{v}(t))$, in the set
\begin{equation}
\label{eqn:calR}
\calR:=\left\{(x_1,\ldots,x_N,v_1,\ldots,v_N)\in\R^{2Nd}: |x_i-x_j|>\lambda \text{ if } i\neq j, |v_i|>\mu\right\},
\end{equation}
for fixed $\lambda,\mu>0$.

The implicit equation \eqref{eqn:vi-bz} for $v_i$ does not depend on
the velocities $v_j$ of the other particles $j \neq i$.  This
motivates the definition of
\[
\mathcal{F}_i:\calRtilde\to\R^d,
\]
\begin{equation}
\mathcal{F}_i(\mathbf{x},v)=\mathcal{F}_i(x_1,\ldots,x_N,v):=-v -\dfrac1N\displaystyle\sum_{j\neq i}\,\nabla_{x_i}K(|x_i-x_j|)\, g\left(\dfrac{(x_i-x_j)}{|x_i-x_j|}\cdot\dfrac{v}{|v|} \right),\label{eqn F zero in fixed point}
\end{equation}
where
\begin{equation}
\calRtilde:=\left\{(x_1,\ldots,x_N,v)\in\R^{Nd}\times\R^d: |x_i-x_j|>\lambda \text{ if } i\neq j, |v|>\mu\right\},
\end{equation}
for any $\lambda,\mu\geq0$.  For a given configuration $\mathbf{x}$,
the velocity $v_i$ is among the zeros of
$\mathcal{F}_i(\mathbf{x},\cdot)$ regarded as a function of $v$.
The implicit function theorem implies immediately the following.
\begin{theorem}[Local continuity]\label{theorem local continuity of trajectories}
Assume that at time $\tau$ the phase space configuration $(\mathbf{x}(\tau),\mathbf{v}(\tau))\in\calR$, with $\mathcal{F}_i(\mathbf{x}(\tau),v_i(\tau))=0$ for all $i\in\{1,\ldots,N\}$, satisfies
\[
\det D_v\mathcal{F}_i(\mathbf{x}(\tau),v_i(\tau)) \neq0 \quad \textrm{ for all } i\in\{1,\ldots,N\}.
\]
Then there is a $\Delta\tau>0$ such that the system
\begin{equation}
\label{eqn:fo-short}
\left\{
  \begin{array}{l}
    \dfrac{d\mathbf{x}}{dt} = \mathbf{v},\\\\
    \mathcal{F}_i(\mathbf{x},v_i)=0, \text{ for all }i\in\{1,\ldots,N\},
  \end{array}
\right.
\end{equation}
has a unique (local) solution $(\mathbf{x},\mathbf{v}):(\tau-\Delta\tau,\tau+\Delta\tau)\to\calR$ that is continuous and that passes through $(\mathbf{x}(\tau),\mathbf{v}(\tau))$ at time $\tau$.
\end{theorem}
 \begin{proof}
 The proof is elementary. From the implicit function theorem, for each $i\in\{1,\ldots,N\}$ there exists an open set $W_i$ and a unique map $\gamma_i\in C^1(W_i;\R^d)$, such that $\mathbf{x}(\tau)\in W_i$, $v_i(\tau)=\gamma_i(\mathbf{x}(\tau))$, and $\mathcal{F}_i(\mathbf{x},\gamma_i(\mathbf{x}))=0$ for all $\mathbf{x}\in W_i$. Define $\Gamma(\mathbf{x}):=(\gamma_1(\mathbf{x}),\ldots,\gamma_N(\mathbf{x}))$ on a closed bounded subset of $\Omega:=\bigcap_{i=1}^NW_i$.  Since $\Gamma$ is  $C^1$, it  is Lipschitz continuous, and the theorem follows from the Picard-Lindelöf Theorem (cf.~\cite[Theorem 2.2]{Teschl}) applied to the system  \begin{equation*}
     \dfrac{d\mathbf{x}}{dt} = \Gamma(\mathbf{x}).
 \end{equation*}
 \end{proof}
\begin{remark} \label{remark:unique}
Given a space configuration $\mathbf{x}^\ast$ and a corresponding
velocity $\mathbf{v}^\ast=(v_1^\ast,\dots,v_N^\ast)$, i.e.,
$\mathcal{F}_i(\mathbf{x}^\ast,v_i^\ast)=0$ for all
$i\in\{1,\ldots,N\}$,
the non-vanishing determinant condition
\[
\det D_v\mathcal{F}_i(\mathbf{x}^\ast,v_i^\ast) \neq0 \quad \textrm{ for all } i\in\{1,\ldots,N\}
\]
guarantees that the fixed point $\mathbf{v}^\ast$ is {isolated}, and
Theorem \ref{theorem local continuity of trajectories}
provides a unique solution of \eqref{eqn:fo-bz} starting at
configuration $\mathbf{x}^\ast$ in the direction
$\mathbf{v}^\ast$. There could be multiple
velocities~$\mathbf{v}^\ast$ corresponding to the same configuration
$\mathbf{x}^\ast$ but as long as such a velocity vector is isolated,
there exists a unique continuous trajectory through $\mathbf{x}^\ast$ in its
direction. This will be revisited in
Section~\ref{sect:conv} in connection with the $\ep \to 0$ limit of
the relaxation system \eqref{eqn:so}.
\end{remark}

\begin{remark} \label{remark:extend}
The (local) continuous solutions provided by Theorem \ref{theorem
  local continuity of trajectories} can be extended in time for as
long as we do not encounter collisions or particle stopping (see
definition \eqref{eqn:calR} of  $\calR$) {\em and} the Jacobian
matrices $D_v\mathcal{F}_i$ remain invertible along the
trajectory. Ruling out collisions and stopping, we conclude that model
\eqref{eqn:fo-bz} has a unique solution that is continuous in position
and velocity \textit{up to the moment} when $\det D_v\mathcal{F}_i=0$
for some $i$. Numerical experiments in Section \ref{sect:breakdown} show that, in the absence of collisions or stopping, discontinuities in velocities occur indeed at such times. To deal with such velocity jumps, both analytically and numerically,  we resort to the relaxation model \eqref{eqn:so} (Sections \ref{sect:conv} and \ref{sect:breakdown}).
 \end{remark}


\subsubsection*{The two-dimensional case}\label{sec intro F R and dF/dtheta}

We now apply the above considerations above to two dimensions to show that
the non-zero determinant condition can be reduced to a very simple
scalar form.
Assume that the configuration $\{x_1,\ldots,x_N\}\subset\R^2$ is
given, and that we search for a nonzero solution of \eqref{eqn:vi-bz}.
Using the polar coordinate representation
$v_i=r_i[\cos\theta_i,\sin\theta_i]^T$, we write \eqref{eqn:vi-bz} as
\begin{equation}\label{eqn fixed point polar coordinates}
r_i{\cos\theta_i\brack\sin\theta_i} = -\dfrac1N\displaystyle\sum_{j\neq i}\nabla_{x_i}K(|x_i-x_j|)\,g\left(\dfrac{x_i-x_j}{|x_i-x_j|}\cdot{\cos\theta_i\brack\sin\theta_i}\right).
\end{equation}
Taking the inner product with $[-\sin\theta_i,\cos\theta_i]^T$ and
$[\cos\theta_i,\sin\theta_i]^T$, the vector equation \eqref{eqn fixed
  point polar coordinates} can be written as
\begin{equation}\label{eqn decomposition F R polar coordinates}
\left\{
  \begin{array}{l}
    0 = -\dfrac1N\displaystyle\sum_{j\neq i}\nabla_{x_i}K(|x_i-x_j|)\cdot{-\sin\theta_i\brack\cos\theta_i}\,g\left(\dfrac{x_i-x_j}{|x_i-x_j|}\cdot{\cos\theta_i\brack\sin\theta_i}\right),
    \\
    r_i = -\dfrac1N\displaystyle\sum_{j\neq i}\nabla_{x_i}K(|x_i-x_j|)\cdot{\cos\theta_i\brack\sin\theta_i}\,g\left(\dfrac{x_i-x_j}{|x_i-x_j|}\cdot{\cos\theta_i\brack\sin\theta_i}\right).
  \end{array}
\right.
\end{equation}
The advantage of the polar coordinates is that the first equation in
\eqref{eqn decomposition F R polar coordinates} is for $\theta_i$ {\em
  only}. Define the following functions ($i=1,\dots,N$):
\begin{subequations}
\label{eqn:FR}
\begin{align}
H_i(\theta)&=-\dfrac1N\displaystyle\sum_{j\neq i}\nabla_{x_i}K(|x_i-x_j|)\cdot{-\sin\theta \brack\cos\theta}\,g\left(\dfrac{x_i-x_j}{|x_i-x_j|}\cdot{\cos\theta \brack\sin\theta}\right),\\
R_i(\theta)&=-\dfrac1N\displaystyle\sum_{j\neq i}\nabla_{x_i}K(|x_i-x_j|)\cdot{\cos\theta \brack\sin\theta}\,g\left(\dfrac{x_i-x_j}{|x_i-x_j|}\cdot{\cos\theta \brack\sin\theta}\right).
\end{align}
\end{subequations}

Hence, solving for $\theta_i$ and $r_i$ from \eqref{eqn decomposition
  F R polar coordinates} is equivalent to finding a root $\theta_i$ of $H_i$ and then setting $r_i$ explicitly:
 \begin{equation}
 \label{eqn:HRroots}
 H_i(\theta_i)=0, \qquad r_i = R_i(\theta_i).
 \end{equation}
Note that a root $\theta_i$ of $H_i$  generates a (non-zero)
admissible velocity if $R_i(\theta_i)>0$.

For a fixed $\mathbf{x}$, we introduce the notation
\begin{equation}
\label{eqn:Fpolar}
\tilde{\mathcal{F}}_i\left(r,\theta\right):=\mathcal{F}_i\left(\mathbf{x},r{\cos\theta\brack\sin\theta}\right)=\left(
\begin{array}{cc}
\cos\theta & -\sin\theta \\
\sin\theta & \cos\theta \\
\end{array}
\right) \cdot
{-r+R_i(\theta) \brack H_i(\theta)},
\end{equation}
with $\mathcal{F}_i$ defined in \eqref{eqn F zero in fixed
  point}.
The chain rule  yields
\begin{equation}
\label{eqn matrix product Jacobians coordinate transform}
D_{(r,\theta)}\tilde{\mathcal{F}}_i = D_v\mathcal{F}_i\cdot D_{(r,\theta)}v,
\end{equation}
where $D_{(r,\theta)}v$ is the Jacobian matrix of the coordinate transform.
Differentiating in \eqref{eqn:Fpolar}, we find
that the second column of $D_{(r,\theta)}\tilde{\mathcal{F}}_i$ is
\begin{equation}\label{eqn second column Jacobian F}
\pd{\tilde{\mathcal{F}}_i}{\theta}=
\left(
\begin{array}{cc}
\cos\theta & -\sin\theta \\
\sin\theta & \cos\theta \\
\end{array}
\right) \cdot
{R'_i(\theta) \brack H'_i(\theta)} +
\left(
\begin{array}{cc}
-\sin\theta & -\cos\theta \\
\cos\theta & -\sin\theta \\
\end{array}
\right) \cdot
{-r+R_i(\theta) \brack H_i(\theta)}.
\end{equation}
Let $(\mathbf{x}^\ast,v^\ast)\in\calRtilde$ satisfy
$\mathcal{F}_i(\mathbf{x}^\ast,v^\ast)=0$, so that
$\tilde{\mathcal{F}}_i\left(r^\ast,\theta^\ast\right)=0$, whence
\begin{equation}\label{eqn R F = 0 polar}
{-r^\ast+R_i(\theta^\ast) \brack H_i(\theta^\ast)}=0.
\end{equation}
Thus, we have
\begin{equation}\label{eqn F polar Jacobian is matrix times matrix R' F'}
D_{(r,\theta)}\tilde{\mathcal{F}}_i =
\left(
\begin{array}{cc}
\cos\theta^\ast & -\sin\theta^\ast \\
\sin\theta^\ast & \cos\theta^\ast \\
\end{array}
\right)
\left(
\begin{array}{cc}
-1 & R'_i(\theta^\ast) \\
0 & H'_i(\theta^\ast) \\
\end{array}
\right).
\end{equation}
Finally, taking the determinant on both sides
of \eqref{eqn matrix product Jacobians coordinate transform} and using
\eqref{eqn F polar Jacobian is matrix times matrix R' F'}, we obtain
\begin{equation}
|H'_i(\theta^\ast)|=r^\ast\,|\det D_v\mathcal{F}_i(\mathbf{x}^\ast,v^\ast)|.
\end{equation}
The condition $\det D_v\mathcal{F}_i(\mathbf{x}^\ast,v^\ast)\neq0$ is
thus equivalent to $H'_i(\theta^\ast)\neq0$. In other words, in two
dimensions, the continuity issues are only to be expected either when
$H'_i$ becomes zero, or when trajectories reach the boundary of
$\calR$ (particles collide or one of the velocities reaches zero).


\section{Relaxation model \eqref{eqn:so}: convergence for $\ep\to0$}
\label{sect:conv}

In this section we investigate the relaxation system
\eqref{eqn:so}. We note that this system is locally well-posed, and
explain in what sense solutions of \eqref{eqn:so} converge to
those of \eqref{eqn:fo-bz} as $\ep\to 0$.


\subsection{Convergence of solutions as $\ep \to 0$}
\label{subsect:conv}

As opposed to \eqref{eqn:fo-bz}, the regularized system
\eqref{eqn:so} has {\em unique} solutions (locally), for each
$\ep>0$, provided that $K'$ and $g$ are bounded and Lipschitz continuous.
We now apply the theory developed by Tikhonov
\cite{Tikhonov1952,Vasileva1963} to study the limit $\ep \to 0$ of
solutions to \eqref{eqn:so}.  We start by paraphrasing some of the
results presented in \cite{Vasileva1963}.
Consider the system of equations
\begin{equation}
\label{eqn shorthand system of equations second-order}
\left\{
  \begin{array}{l}
    \dfrac{d\mathbf{x}}{dt} = \mathbf{v},\\\\
    \ep\dfrac{d\mathbf{v}}{dt} = \mathcal{F}(\mathbf{x},\mathbf{v}),
  \end{array}
\right.
\end{equation}
where $\mathbf{x},\mathbf{v} \in \mathbb{R}^{Nd}$ and $\ep>0$ is a small parameter.
On a closed and bounded set $D\subset\R^{Nd}$,
 let $\Gamma:D\to\R^{Nd}$ be such that $\mathbf{v}=\Gamma(\mathbf{x})$ is a solution of the system of equations
\begin{equation}\label{eqn F implicit}
\mathcal{F}(\mathbf{x},\mathbf{v})=0.
\end{equation}
The function $\Gamma$ is called a \textit{root} of \eqref{eqn F implicit}. The system
\begin{equation}\label{eqn shorthand system of equations first-order}
\left\{
  \begin{array}{l}
    \dfrac{d\mathbf{x}}{dt} = \mathbf{v},\\\\
    \mathbf{v} = \Gamma(\mathbf{x}),
  \end{array}
\right.
\end{equation}
is called \textit{the degenerate system of equations corresponding to the root $\mathbf{v}=\Gamma(\mathbf{x})$} .

Note that the systems of our interest \eqref{eqn:fo-bz} and \eqref{eqn:so} can be written in the short-hand notation \eqref{eqn shorthand system of equations first-order} and \eqref{eqn shorthand system of equations second-order}, respectively. Indeed, define $\mathcal{F}:\calR\to\R^{Nd}$ using \eqref{eqn F zero in fixed point} as
\begin{equation}\label{eqn def cal F concatenation cal Fi}
\mathcal{F}(\mathbf{x},\mathbf{v}):=\left(\mathcal{F}_1(\mathbf{x},v_1),\ldots,\mathcal{F}_N(\mathbf{x},v_N)\right),
\end{equation}
for all $(\mathbf{x},\mathbf{v})\in\calR$. Then, \eqref{eqn:so} can be written compactly as \eqref{eqn shorthand system of equations second-order}, and  \eqref{eqn:fo-bz} is a {\em degenerate} system in the form  \eqref{eqn shorthand system of equations first-order}, with function $\Gamma =(\gamma_1,\dots,\gamma_N)$ provided by the implicit function theorem (see Theorem \ref{theorem local continuity of trajectories} and its proof).

\begin{definition}[Isolated root]
The root $\Gamma$ is called \textit{isolated} if there is a $\delta>0$ such that for all $\mathbf{x}\in D$ the only element in $B(\Gamma(\mathbf{x}),\delta)$ that satisfies $\mathcal{F}(\mathbf{x},\mathbf{v})=0$ is $\mathbf{v}=\Gamma(\mathbf{x})$.
\end{definition}

\begin{definition}[Adjoined system and positive stability]\label{def pos stab}
For fixed $\mathbf{x}^\ast$, the system
\begin{equation}\label{eqn adjoined system}
\dfrac{d\mathbf{v}}{d\tau}=\mathcal{F}(\mathbf{x}^\ast,\mathbf{v}),
\end{equation}
is called the \textit{adjoined system of equations}. An isolated root $\Gamma$ is called \textit{positively stable} in $D$, if $\mathbf{v}^\ast=\Gamma(\mathbf{x}^\ast)$ is an asymptotically stable stationary point of \eqref{eqn adjoined system} as $\tau\to\infty$, for each $\mathbf{x}^\ast\in D$.
\end{definition}

\begin{definition}[Domain of influence]
The \textit{domain of influence} of an isolated positively stable root $\Gamma$ is the set of points $(\mathbf{x}^\ast,\tilde{\mathbf{v}})$ such that the solution of \eqref{eqn adjoined system} satisfying $\mathbf{v}|_{\tau=0}=\tilde{\mathbf{v}}$ tends to $\mathbf{v}^\ast = \Gamma(\mathbf{x}^\ast)$ as $\tau\to\infty$.
\end{definition}
The following theorem, due to Tikhonov  \cite{Tikhonov1952},  states under which conditions and in what sense solutions of \eqref{eqn shorthand system of equations second-order} converge to solutions of the (degenerate) system \eqref{eqn shorthand system of equations first-order}.

\begin{theorem}[see \cite{Tikhonov1952} or \cite{Vasileva1963}, Thm. 1.1] \label{thm:Tikhonov}
Assume that $\Gamma$ is an isolated positively stable root of \eqref{eqn F implicit} in some bounded closed domain $D$. Consider a point $(\mathbf{x}_0,\mathbf{v}_0)$ in the domain of influence of this root, and assume that the degenerate system  \eqref{eqn shorthand system of equations first-order} has a solution
 $\mathbf{x}(t)$ initialized at $\mathbf{x}(t_0) = \mathbf{x}_0$, that lies in $D$ for all $t\in[t_0,T]$. Then, as $\ep\to 0$, the solution
$(\mathbf{x}^\ep (t),\mathbf{v}^\ep (t))$
of \eqref{eqn shorthand system of equations second-order} initialized at $(\mathbf{x}_0,\mathbf{v}_0)$, converges to $(\mathbf{x}(t),\mathbf{v}(t)):=(\mathbf{x}(t),\Gamma(\mathbf{x}(t)))$ in the following sense:

(i) $\displaystyle\lim_{\ep\to 0}\mathbf{v}^\ep(t)=\mathbf{v}(t)$ \text{ for all } $t\in(t_0,T^*]$, \text {and }

(ii) $\displaystyle\lim_{\ep\to 0}\mathbf{x}^\ep(t)=\mathbf{x}(t)$ \text{ for all } $t\in[t_0,T^*]$,

for some $T^*<T$.
\end{theorem}

\begin{remark}\label{remark:boundary layer}
The degenerate system requires an initial condition $\mathbf{x}_0$ only for positions, while for the  $\ep$-system both $\mathbf{x}_0$ and $\mathbf{v}_0$ need to be provided. It is possible that $\mathbf{v}_0$ is \textit{incompatible} in the sense that $\mathbf{v}_0\neq\Gamma(\mathbf{x}_0)$. This is exactly why the convergence of $\mathbf{v}^\ep(t)$ to $\mathbf{v}(t)$ only holds for $t>t_0$. In case of incompatible initial conditions an initial boundary layer forms, which gets narrower as $\ep\to0$.
\end{remark}
Theorem \ref{thm:Tikhonov} can now be used to infer convergence of
solutions of \eqref{eqn:so} to solutions of~\eqref{eqn:fo-bz}.
\begin{theorem}[Convergence of the relaxation model]
\label{thm:conv}
Assume that the isolated root $\Gamma$ is {\em positively stable} in
$D$, and take $(\mathbf{x}_0,\mathbf{v}_0)$ in the domain of influence
of this root. Denote by $\mathbf{x}(t)$ the (local) solution in $D$ of
the degenerate system \eqref{eqn:fo-bz} with initial configuration
$\mathbf{x}_0$ (the existence of this solution is provided by Theorem
\ref{theorem local continuity of trajectories}). Then, the solution
$(\mathbf{x}^\ep (t),\mathbf{v}^\ep (t))$ of the regularized system
\eqref{eqn:so}, initialized at $(\mathbf{x}_0,\mathbf{v}_0)$,
converges as $\ep \to 0$ to
$(\mathbf{x}(t),\mathbf{v}(t)):=(\mathbf{x}(t),\Gamma(\mathbf{x}(t)))$
in the sense i) and ii) given in Theorem \ref{thm:Tikhonov}.
\end{theorem}


\begin{remark} \label{remark:breakdown} Cf. Remark
  \ref{remark:extend}, unless collisions or stopping occur, a $C^1$
  solution $\mathbf{x}(t)$ of \eqref{eqn:fo-bz} exists as long as
  $\det D_v\mathcal{F}_i(\mathbf{x}(t),v_i(t))\neq0$ for all
  $i=1,\dots,N$. Positive stability of
  $\mathbf{v}(t)=\Gamma(\mathbf{x}(t))$ is equivalent to eigenvalues
  of $D_{v}\mathcal{F}_i(\mathbf{x}(t),v_i(t))$ to be negative along
  the trajectory, for all $i=1,\ldots,N$. Moreover, once all these
  eigenvalues are negative at the initial time, they remain negative
  through the domain of existence of $\mathbf{x}(t)$, because none of
  these eigenvalues can change sign before $\det
  D_v\mathcal{F}_i(\mathbf{x}(t),v_i(t))$ touches $0$ for some
  $i$. Hence, we infer that the convergence in Theorem \ref{thm:conv}
  applies in all smooth regions of solutions $\mathbf{x}(t)$ of
  \eqref{eqn:fo-bz}, before a breakdown of the solution occurs.
\end{remark}

\begin{remark}
Theorem \ref{thm:conv} (trivially) implies the convergence result for the isotropic case $g\equiv1$ as $\ep\to0$. To our knowledge, this has not been stated clearly by any previous work on this model.
\end{remark}


\subsection{Positive stability of roots in two dimensions}
\label{sec:pos stab d=2}

Next, we elaborate the above convergence result with an example in
dimension $d=2$. In particular, we show that the stability of the
fixed point $\Gamma$ is essential, as otherwise the convergence fails.
The notion of asymptotic stability in Definition~\ref{def pos stab}
should be understood in the sense of Lyapunov. A stationary point
$\mathbf{v}^\ast=\Gamma(\mathbf{x}^\ast)$ of~\eqref{eqn adjoined
  system} is asymptotically stable, if and only if all eigenvalues of
$D_\mathbf{v}\mathcal{F}(\mathbf{x}^\ast,\mathbf{v}^\ast)$ have
strictly negative real part. Due to \eqref{eqn def cal F concatenation
  cal Fi}, the set of eigenvalues of $D_\mathbf{v}\mathcal{F}$ equals
to the union of the eigenvalues of all $D_{v}\mathcal{F}_i$,
$i=1,\ldots,N$.  Let~$\mathbf{v}^\ast = (v_1^\ast,\dots,v_N^\ast)$ be
a stationary point of \eqref{eqn adjoined system}, that is,
$\mathcal{F}_i(\mathbf{x}^\ast,v_i^\ast)=0$ for all $i=1,\dots,N$. We
write each velocity $v_i^\ast$ in the polar coordinates,
$v_i^\ast=r_i^\ast[\cos\theta_i^\ast,\sin\theta_i^\ast]^T$. To compute
the eigenvalues of $D_{v}\mathcal{F}_i(\mathbf{x}^\ast,v_i^\ast)$, we use
\eqref{eqn matrix product Jacobians coordinate transform} and
\eqref{eqn F polar Jacobian is matrix times matrix R' F'}, to get
\begin{equation}\label{eqn F' R' related to DvF via M Mr}
\underbrace{\left(
\begin{array}{cc}
\cos\theta_i^\ast & -\sin\theta_i^\ast \\
\sin\theta_i^\ast & \cos\theta_i^\ast \\
\end{array}
\right)}_{=:M}
\left(
\begin{array}{cc}
-1 & R'_i(\theta_i^\ast) \\
0 & H'_i(\theta_i^\ast) \\
\end{array}
\right)
=
D_v\mathcal{F}_i (\mathbf{x}^\ast,v_i^\ast)
\underbrace{\left(
\begin{array}{cc}
\cos\theta_i^\ast & -r_i^\ast \sin\theta_i^\ast\\
\sin\theta_i^\ast & r_i^\ast \cos\theta_i^\ast\\
\end{array}
\right)}_{=:M_r} .
\end{equation}
Note that the functions $H_i$ and $R_i$ used here (see \eqref{eqn:FR})
correspond to the fixed spatial configuration $\mathbf{x}^\ast$. The
matrices $M_r^{-1}D_v\mathcal{F}_i M_r$ and $D_v\mathcal{F}_i$ have
the same set of eigenvalues,
and hence, we conclude from \eqref{eqn F' R' related to DvF via M Mr}
that $D_v\mathcal{F}_i(\mathbf{x}^\ast,v_i^\ast)$ has only eigenvalues
with negative real part, if and only if this is the case for
\begin{equation*}
M_r^{-1}\,M\,\left(
\begin{array}{cc}
-1 & R'_i(\theta_i^\ast) \\
0 & H'_i(\theta_i^\ast) \\
\end{array}
\right).
\end{equation*}
We have
\begin{equation*}
M_r^{-1}\,M =
\left(
\begin{array}{cc}
1 & 0 \\
0 & 1/r_i^\ast \\
\end{array}
\right),
\end{equation*}
and the eigenvalues of
\begin{equation*}
\left(
\begin{array}{cc}
1 & 0 \\
0 & 1/r_i^\ast \\
\end{array}
\right)
\left(
\begin{array}{cc}
-1 & R'_i(\theta_i^\ast) \\
0 & H'_i(\theta_i^\ast) \\
\end{array}
\right)
\end{equation*}
are
\begin{equation}
\label{eqn:evalues}
\lambda_1=-1 \quad  \text { and }  \quad \lambda_2=H'_i(\theta_i^\ast)/r_i^\ast.
\end{equation}
Note that (within $\calR$) all eigenvalues are real-valued.
Therefore, in view of \eqref{eqn:evalues}, $\mathbf{v}^\ast = \Gamma(\mathbf{x}^\ast)$ is asymptotically stable provided
\begin{equation}
\label{eqn:asy_stab}
H'_i(\theta_i^\ast)<0 \quad \text { for all } \quad i\in\{1,\ldots,N\}.
\end{equation}

\begin{remark}
An even more direct way of reaching \eqref{eqn:asy_stab} is to express the adjoint system \eqref{eqn adjoined system} in polar coordinates. Indeed, for each index $i$, \eqref{eqn adjoined system} yields
\begin{equation}
\label{eqn:adj-i}
\dfrac{d v_i}{d\tau}=\mathcal{F}_i(\mathbf{x}^\ast,v_i).
\end{equation}
Write $v_i$ in polar coordinates $v_i=r_i [\cos\theta_i,\sin\theta_i]^T$ and use \eqref{eqn F zero in fixed point} and notations \eqref{eqn:FR} (with functions $H_i$ and $R_i$ corresponding to the spatial configuration  $\mathbf{x}^\ast = (x_1^\ast,\dots,x_N^\ast)$), to derive from \eqref{eqn:adj-i}:
\begin{subequations}
\label{eqn:polars}
\begin{align}
    \frac{d\theta_i}{d \tau} &= \frac{1}{r_i} H_i(\theta_i), \label{eqn:polars-theta}\\
    \frac{dr_i}{d \tau} &= -r_i + R_i(\theta_i).
  \end{align}
\end{subequations}
The condition \eqref{eqn:asy_stab} for the asymptotic stability can
then be seen directly from \eqref{eqn:polars}. Indeed, the
linearization of \eqref{eqn:polars} around the stationary point
$(r_i^\ast,\theta_i^\ast)$ yields the Jacobian matrix
\[
\left(
\begin{array}{cc}
H_i'(\theta_i^\ast)/{r_i^\ast} & 0 \\
R'_i(\theta_i^\ast) &  -1 \\
\end{array}
\right),
\]
with the eigenvalues given by \eqref{eqn:evalues}.
\end{remark}

 \begin{remark}
 \label{remark:summary}
 We note that using the polar coordinates in two dimensions reduces the
 calculations to {\em scalar} expressions. For a better clarification of
 this point, let us summarize the findings so far. For convenience of
 notations, we drop the $\ast$ superscript.

 Consider a given spatial configuration $\mathbf{x} = (x_1,\dots,x_N)$. Then the following hold.
 \begin{itemize}
 \item To find the velocities $v_i$ corresponding to this configuration (solve $\mathcal{F}_i(\mathbf{x},v_i) =0$), it is more convenient to use polar coordinates $v_i=r_i[\cos \theta_i, \sin \theta_i]^T$. The problem reduces to finding the roots $\theta_i$ of $H_i(\theta)=0$. Then take $r_i= R_i(\theta_i)$, $i= 1,\ldots,N$ (for a $\theta_i$ to be admissible, it needs that $R_i(\theta_i)>0)$.

 \item The condition $\det D_v\mathcal{F}_i(\mathbf{x},v_i)\neq0$ for all $i=1,\dots,N$ guarantees that the fixed point $\mathbf{v}=\Gamma(\mathbf{x})$ is isolated and that  \eqref{eqn:fo-bz}  has a unique continuous solution through $\mathbf{x}$ in the direction $\mathbf{v}$ (see Remark \ref{remark:unique}). In polar coordinates this condition is equivalent to $H'_i(\theta_i)\neq0$, that is, $\theta_i$ is a {\em simple} root of $H_i$ for all $i=1,\dots,N$.

 \item The fact that the isolated root $\mathbf{v}=\Gamma(\mathbf{x})$ is positively stable, as required for the convergence of the $\ep$-regularization (see Theorem \ref{thm:conv}), is equivalent to $H'_i(\theta_i)<0$ for all $i=1,\ldots,N$.
 \end{itemize}
 \end{remark}


\paragraph{A numerical example.}
\label{ex:instable}
We start by noting that in {\em all} numerical experiments presented in this paper we use the same choices of the potential $K$ and field-of-vision function $g$. For the potential $K$ we take the Morse potential \eqref{eqn:Morse} with $C_a=3, C_r=2, l_a=2, l_r=1$. The function $g$ for the field of vision is taken as in \eqref{eqn:gtanh} with parameters $a=5, b=\pi$. This choice corresponds to the solid line in Figure \ref{fig:vision}(b). Note that the width of the field of vision is approximately $180^\circ$ (frontal vision).

We present here a numerical example in two dimensions to illustrate the
convergence of the $\ep$ system in Theorem \ref{thm:conv}. We consider
a randomly-generated initial configuration of four particles --- see
Figure \ref{fig:instab}(a). For the top left particle, labeled as
particle~1, we plot the functions $H_1$ and $R_1$ defined by
\eqref{eqn:FR}, and note that there are three admissible initial directions (see \eqref{eqn:HRroots} and Figure
\ref{fig:instab}(b)):
$\theta_1\approx-1.78$, $\theta_1\approx-1.28$, $\theta_1\approx-1.00$, and
they are all simple roots since~$H_1'(\theta_1)\neq 0$.  Consequently,
there are three isolated fixed points $v_1$ that represent the
possible initial velocities for particle~1. The other three particles have unique velocities at this configuration--- see Figure \ref{fig:instab}(a) where  the admissible velocities are indicated by arrows.

\begin{figure}%
\centering
\includegraphics[width=0.43\textwidth]{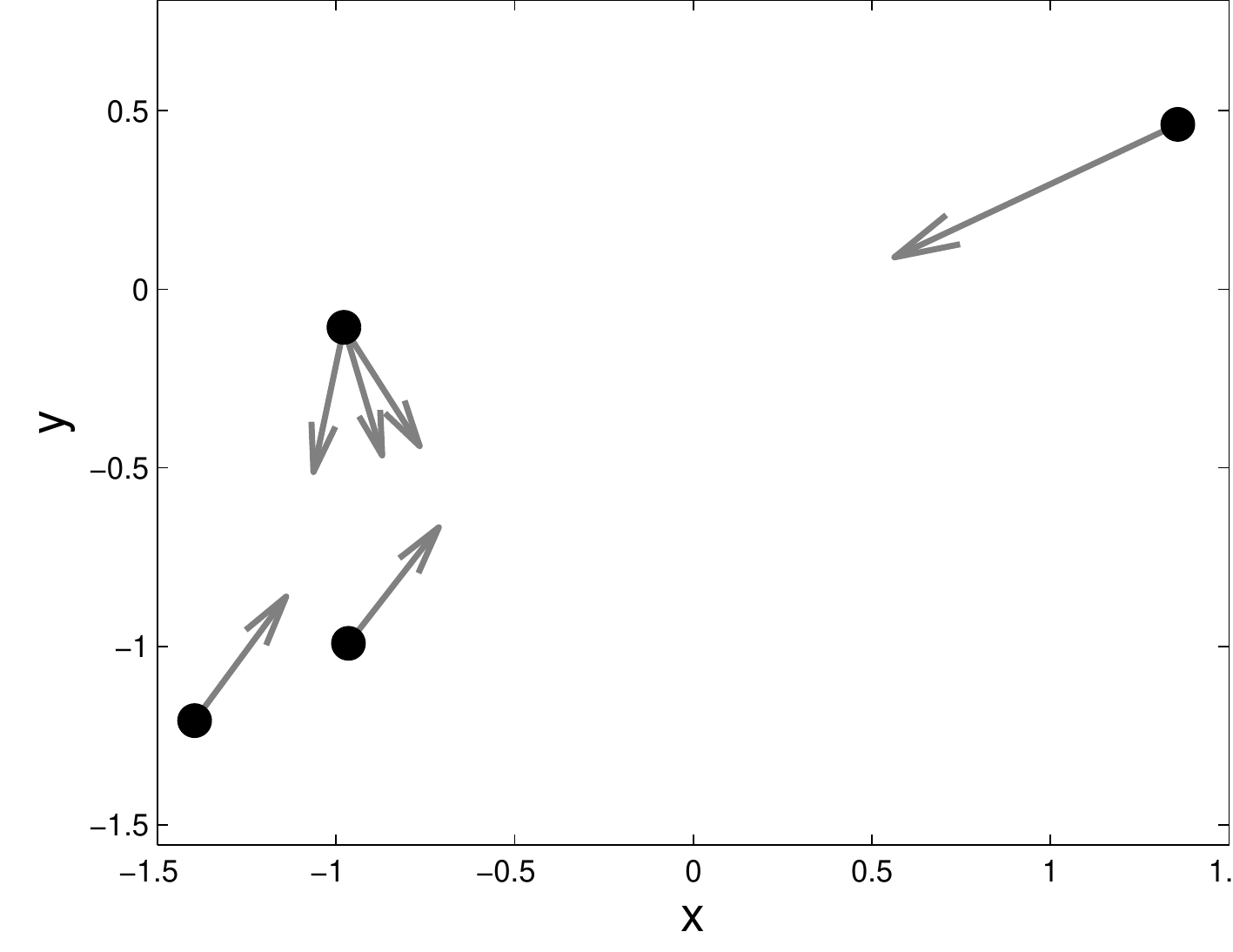}
$~~~$
\includegraphics[width=0.45\textwidth]{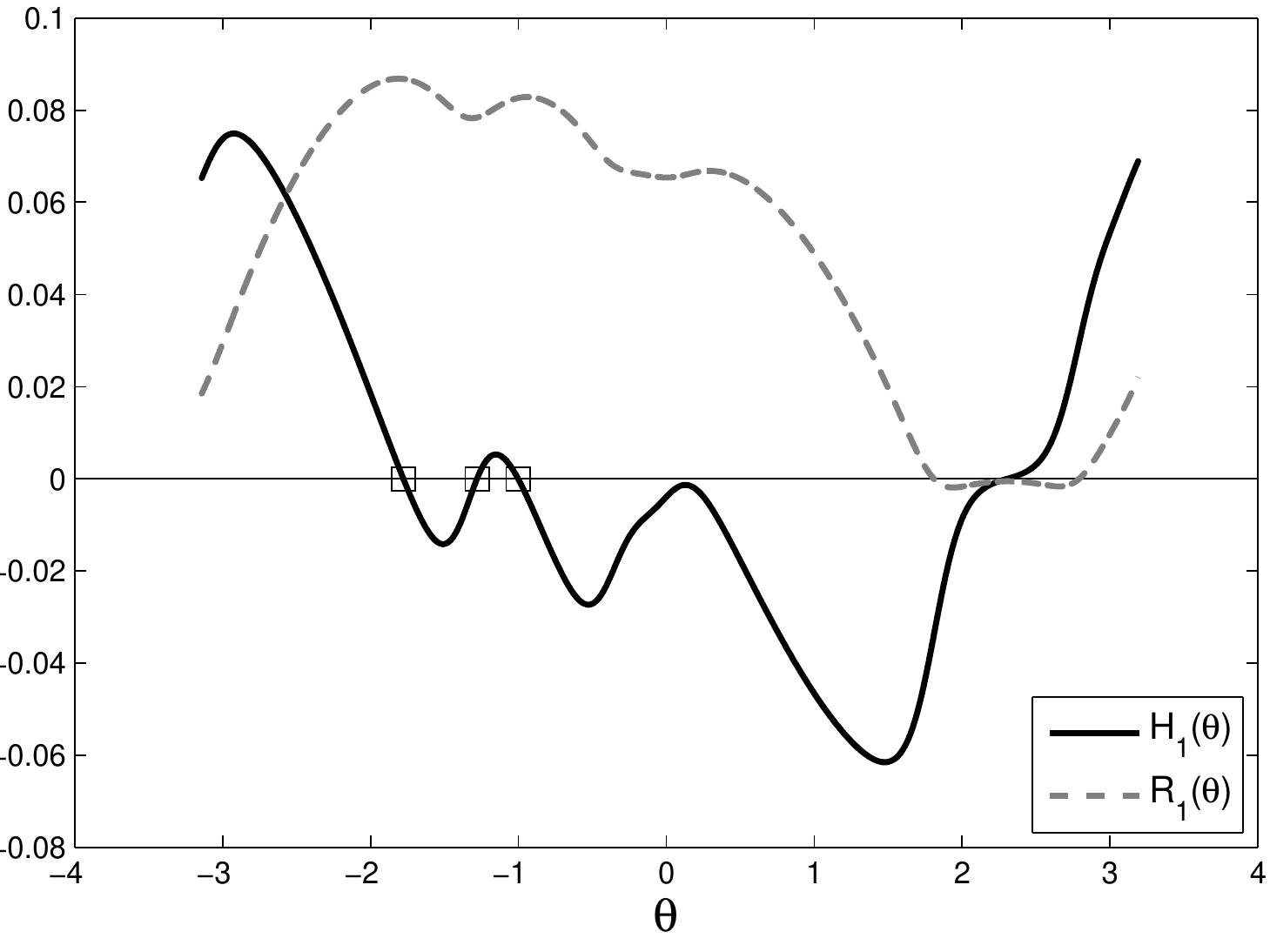}\\
(a)\hspace{0.45\textwidth} (b)\\
\includegraphics[width=0.45\textwidth]{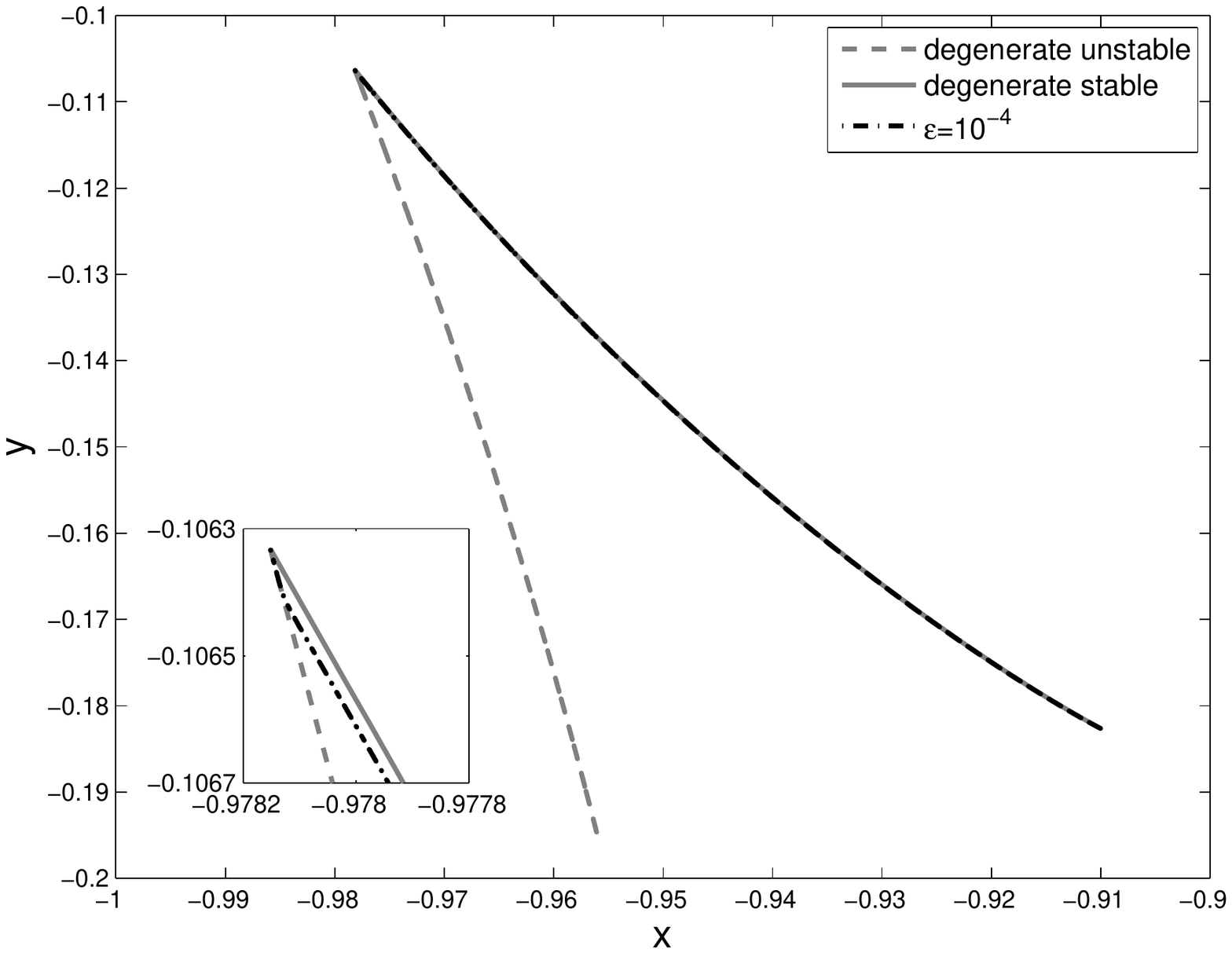}
$~~~$
\includegraphics[width=0.45\textwidth]{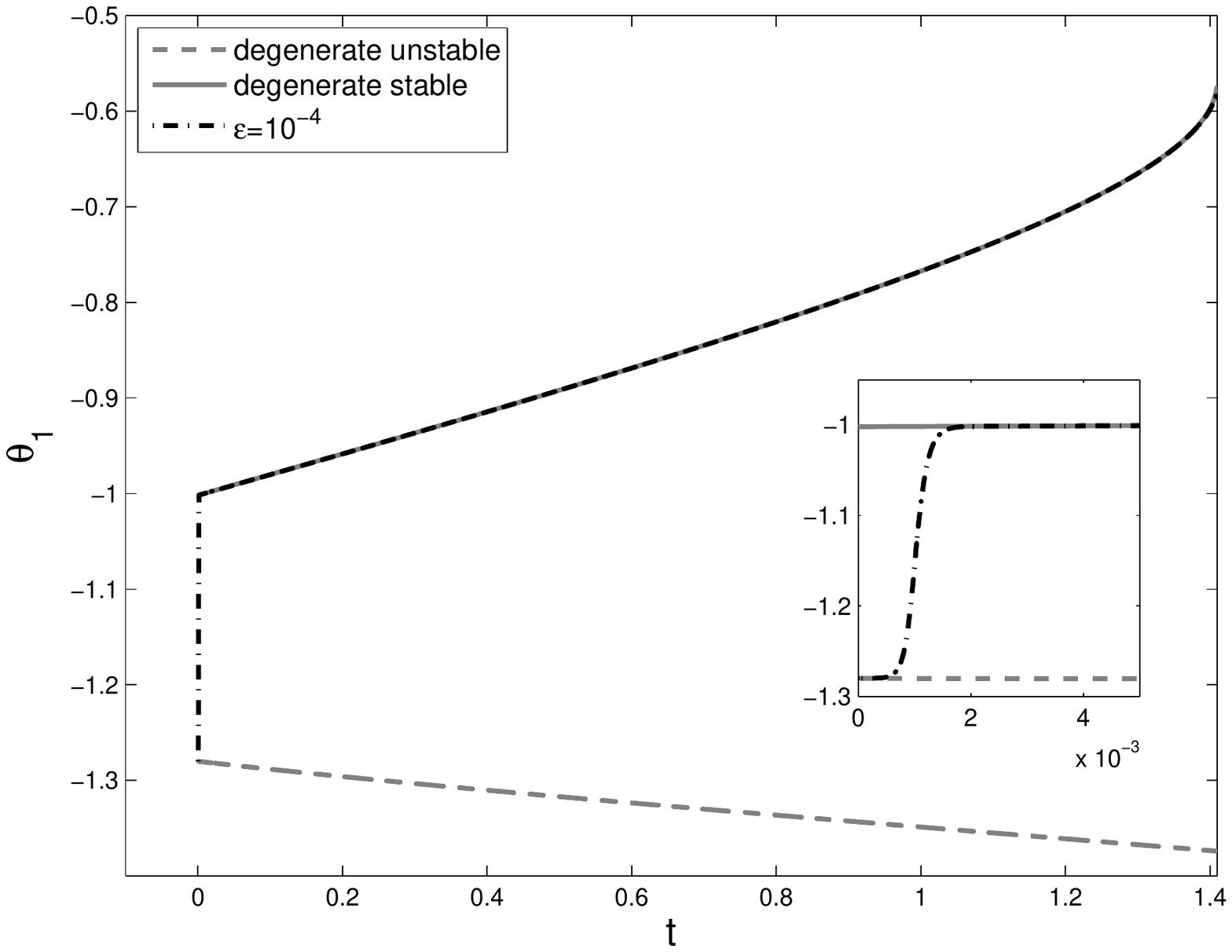}\\
(c)\hspace{0.45\textwidth} (d)
\caption{Illustration of the $\ep \to 0$ limit. (a) A randomly generated initial configuration of $4$ particles, with indication of their admissible velocities. Only one of them (the top left particle, labeled as particle $1$) allows for multiple initial velocities. (b) Plot of the functions $H_1$ and $R_1$ for particle $1$. There are three admissible values of $\theta_1$ for which $H_1(\theta_1)=0$ and $R_1(\theta_1)>0$, each indicated by a square. All the three roots are simple, resulting in the three {\em isolated} velocities shown on the left. The centre root is unstable, while the other two are stable. (c) Trajectories
$x_1(t)$ of the anisotropic (degenerate) system \eqref{eqn:fo-bz} starting in the directions of the centre root
$\theta_1 \approx -1.28$ and of the right-side root $\theta_1 \approx -1.00$, are shown in grey dashed and grey solid, respectively. The dash-dotted line represents the trajectory $x_1^\ep(t)$ of the $\ep$-system \eqref{eqn:so} with $\ep=10^{-4}$, initialized in the direction of the centre root.  Since this root is unstable, $x_1^\ep$ leaves this direction and approaches via a boundary layer (see the insert) the solution of the degenerate system that corresponds to a stable root (in this case, the right-side root). (d) Same numerical experiment as in (c), but showing the polar angle $\theta_1$ of the velocity $v_1$ as a function of time. The $\ep$-system starts at the unstable root $\theta_1 \approx -1.28$, but relaxes via an initial transition layer (see the insert) to the the solution of the degenerate system that starts at the stable root $\theta_1 \approx -1.00$.}%
\label{fig:instab}%
\end{figure}

Each such $v_1$ corresponds to a continuous trajectory of \eqref{eqn:fo-bz} starting from the initial configuration in Figure \ref{fig:instab}(a). In a numerical implementation one has to {\em pick} one of these admissible initial velocities and then evolve system \eqref{eqn:fo-bz} in time. We use the 4th order Runge-Kutta method for the numerical implementation. Figure \ref{fig:instab}(c) shows the trajectories of particle 1 that correspond to two of these admissible initial velocities: grey dashed (corresponding to root $\theta_1\approx-1.28$) and grey solid (corresponding to $\theta_1\approx-1.00$). Each trajectory in Figure \ref{fig:instab}(c) is the unique continuous solution given by Theorem \ref{theorem local continuity of trajectories} plotted on its {\em maximal} interval of existence --- the possible modes of breakdown are discussed in detail in Section \ref{sect:breakdown}.

We turn now to the convergence of the $\ep$ regularization \eqref{eqn:so} and the role of the positive stability assumption in Theorem \ref{thm:conv}.  Note that at the centre root $\theta_1 \approx-1.28$, $H_1$  has positive slope, while $H_1'<0$ at the other two roots. It means that only the roots at $\theta_1\approx-1.78$ and $\theta_1\approx-1.00$ are positively stable, the centre one is not. The regularized system \eqref{eqn:so} is {\em not} expected to converge to the trajectory corresponding to the centre root  and Figure \ref{fig:instab}(c) illustrates this fact. More specifically, the dash-dotted line shows the trajectory $x_1^\ep(t)$  of particle 1, obtained by integrating numerically \eqref{eqn:so} starting from the configuration in Figure \ref{fig:instab}(a) and an initial velocity that corresponds to the root at $\theta_1\approx-1.28$. Here, $\ep = 10^{-4}$. Note that, due to the instability of this root,  the trajectory of the $\ep$-model does not follow the dashed trajectory of model \eqref{eqn:fo-bz}. Instead, it approaches via a thin initial boundary layer, the solid trajectory of \eqref{eqn:fo-bz} that corresponds to the stable root $\theta_1\approx-1.00$.

We do not address here the question of why the root at $\theta_1\approx-1.00$ was ``chosen", and not the one at $\theta_1\approx-1.78$, since identifying domains of influence of stable roots is a challenge in itself.  We just note that the initial velocity we provided for the $\ep$-system happened to be in the domain of influence of $\theta_1\approx-1.00$. Finally, for an enhanced visualization, the stability/instability of the roots is also illustrated in Figure \ref{fig:instab}(d), which shows the time evolution of the polar angle $\theta_1(t)$ of $v_1(t)$. The dashed grey and solid grey lines represent the evolution $\theta_1(t)$ corresponding to the like-marked trajectories in Figure \ref{fig:instab}(c) (continuous solutions of \eqref{eqn:fo-bz} that correspond to initial $\theta_1\approx-1.28$ and $\theta_1\approx-1.00$, respectively). The black dash-dotted line represents the evolution $\theta_1^\ep(t)$ obtained from \eqref{eqn:so}. Initialized at the unstable root, $\theta_1^\ep$ undergoes through a boundary layer before approaching the solid grey line corresponding to the stable initial root $\theta_1\approx-1.00$.


\section{Breakdown and jump selection}
\label{sect:breakdown}

Smooth solutions to \eqref{eqn:fo-nobz} may cease to exist due to various factors. In this section, we investigate these modes of breakdown and explain how jumps can be meaningfully enforced. We provide numerical illustrations of these ideas in two dimensions.

\subsection{Modes of breakdown: classification}
\label{sect:modes}
A possible breakdown of $C^1$ solutions to \eqref{eqn:fo-bz} was already indicated in previous sections (see Remark \ref{remark:extend}). Namely, it may occur when one of the Jacobian matrices $D_v \mathcal{F}_i$ becomes singular for some particle $i$. At such time, the phase-space trajectory $(x_i(t),v_i(t))$ may cease to be continuous, provided that, for any continuous extension of $\{x_1,\ldots,x_N\}$ in the direction of $\{v_1,\ldots,v_N\}$, there is no zero of \eqref{eqn F zero in fixed point} in a (sufficiently small) neighbourhood of $v_i$. In other words, the {\em current} velocity $v_i$ may cease to be a zero of  \eqref{eqn F zero in fixed point} beyond this time and a jump in $v_i$ has to be enforced. We call such a discontinuity in velocities, due to root losses of $\mathcal{F}_i$, a jump of Type I.

Other modes of breakdown are also possible: collision of particles and stopping. We do not address the former. Collisions are a delicate matter, which has not been properly addressed even in the context of isotropic models. Very briefly, the repulsion component in the interaction potential $K$ has to be strong {\em enough} to counteract the attraction. Particle collisions have been discussed in \cite{BertozziCarilloLaurent} for instance, but the potentials there are purely attractive. For our purpose, we sidestep the issue, and focus instead on particle stopping, that is, when one $v_i=0$. In fact, this mode of breakdown is not present in the isotropic model \eqref{eqn:fo-nobz}, being entirely characteristic to the anisotropic  model \eqref{eqn:fo-bz}.

Note that, as given by \eqref{eqn:fo-bz}, the anisotropic model is not even defined when one particle is at rest ($v_i=0$). This is because the definition of the field of vision assumes the existence of a current direction of motion (an individual facing a certain direction). However, $v_i=0$ can be considered as a solution of \eqref{eqn:vi-bz} in the generalized interpretation of Definition \ref{def generalization fixed point}. And indeed, in numerics, we observe that $v_i=0$ does occur, in a sense that is consistent with this definition. More precisely, we observe numerically that a generic particle $i$ brakes and then stops, in a {\em continuous} fashion, along its direction of motion. One-sided continuity of $v_i/|v_i|$ at the stopping time (called here $t^\ast$) is essential, as this enables us to pass the limit $t \nearrow t^\ast$ in \eqref{eqn:vi-bz} and find that $v_i=0$ is a solution of \eqref{generalized fixed point}, with $s= \lim_{t \nearrow t^\ast} v_i(t)/|v_i(t)|$.

To illustrate the stopping idea in two dimensions, take the polar coordinate representation $v_i(t)=r_i(t) [ \cos \theta_i(t), \sin \theta_i(t)]^T$. Then, by braking continuously and stopping at time $t^\ast$, we mean that:
\[
\lim_{t \nearrow t^\ast} r_i(t) = 0, \qquad \lim_{t \nearrow t^\ast} \theta_i(t) = \theta_i^\ast,
\]
for some angle $\theta_i^\ast$. Hence,  since $v_i(t)/|v_i(t)|= [ \cos \theta_i(t), \sin \theta_i(t)]^T$, equation \eqref{eqn:vi-bz} has a well defined limit $t \nearrow t^\ast$. By passing to the limit we find
\[
0 =  -\dfrac1N\displaystyle\sum_{j\neq i}\,\nabla_{x_i}K(|x_i^\ast-x_j^\ast|)\, g\left(\dfrac{(x_i^\ast-x_j^\ast)}{|x_i^\ast-x_j^\ast|}\cdot  {\cos \theta_i^\ast \brack \sin \theta_i^\ast} \right),
\]
where $\{x_j^\ast\}$ represent the spatial configuration at $t^\ast$. Hence, $v_i=0$ solves \eqref{generalized fixed point} at $t=t^\ast$ with $s=[ \cos \theta_i^\ast, \sin \theta_i^\ast]^T$.

In all numerical experiments we performed, we noticed that particles stop continuously, in the above sense. However, the typical scenario is that there is no continuous phase-space trajectory $(x_i(t),v_i(t))$ for particle $i$ beyond its stopping at time $t^\ast$. Similar to the root loss jump (Type I), $v_i=0$ is a (generalized) solution of \eqref{eqn F zero in fixed point} at $t=t^\ast$, but to evolve the system further in time a jump in $v_i$ has to be enforced. We call the jumps due to particle stopping, jumps of Type II.



We emphasize that throughout this section, by jump discontinuities for \eqref{eqn:fo-bz} we mean jumps in velocities $v_i$, and not in the actual trajectories $x_i$. The latter remain continuous through jumps.


\subsection{Numerical illustrations  in two dimensions}
We illustrate the two modes of breakdown in two dimensions. A breakdown of type I occurs when $H_i'(\theta_i)=0$ for some $i$ at $t=t^\ast$ (see Remark \ref{remark:summary}). Equivalently, $\theta_i$ is no longer a simple root of $H_i$. For a numerical illustration, we reconsider the run of \eqref{eqn:fo-bz}  indicated by solid grey lines in Figures \ref{fig:instab}(c) and (d), that is, the solution that starts in the direction of the stable initial root $\theta_1 \approx -1.00$. At $t^\ast = 1.41$, the current direction $\theta_1 \approx -0.57$ becomes a {\em double} root of $H_1$, as illustrated in  Figure \ref{fig:HRjump1411}. The empty circle represents the root $\theta_1$ just before the jump at $t^\ast$ occurs.
Moreover, there exists no {\em continuous} extension of the phase-space trajectory beyond $t=t^\ast$ (since the double root would disappear and would no longer be a root immediately after $t^\ast$!). The insert in Figure \ref{fig:HRjump1411} illustrates this transition. The solid black line shows the graph of $H_1(\theta)$ before the jump, where the root $\theta_1 \approx -0.57$ is still present. By extending the dynamics in the direction of the current velocity, this root disappears (the dash-dotted line in the insert).

\begin{figure}[tb]%
\centering
\includegraphics[width=0.6\textwidth]{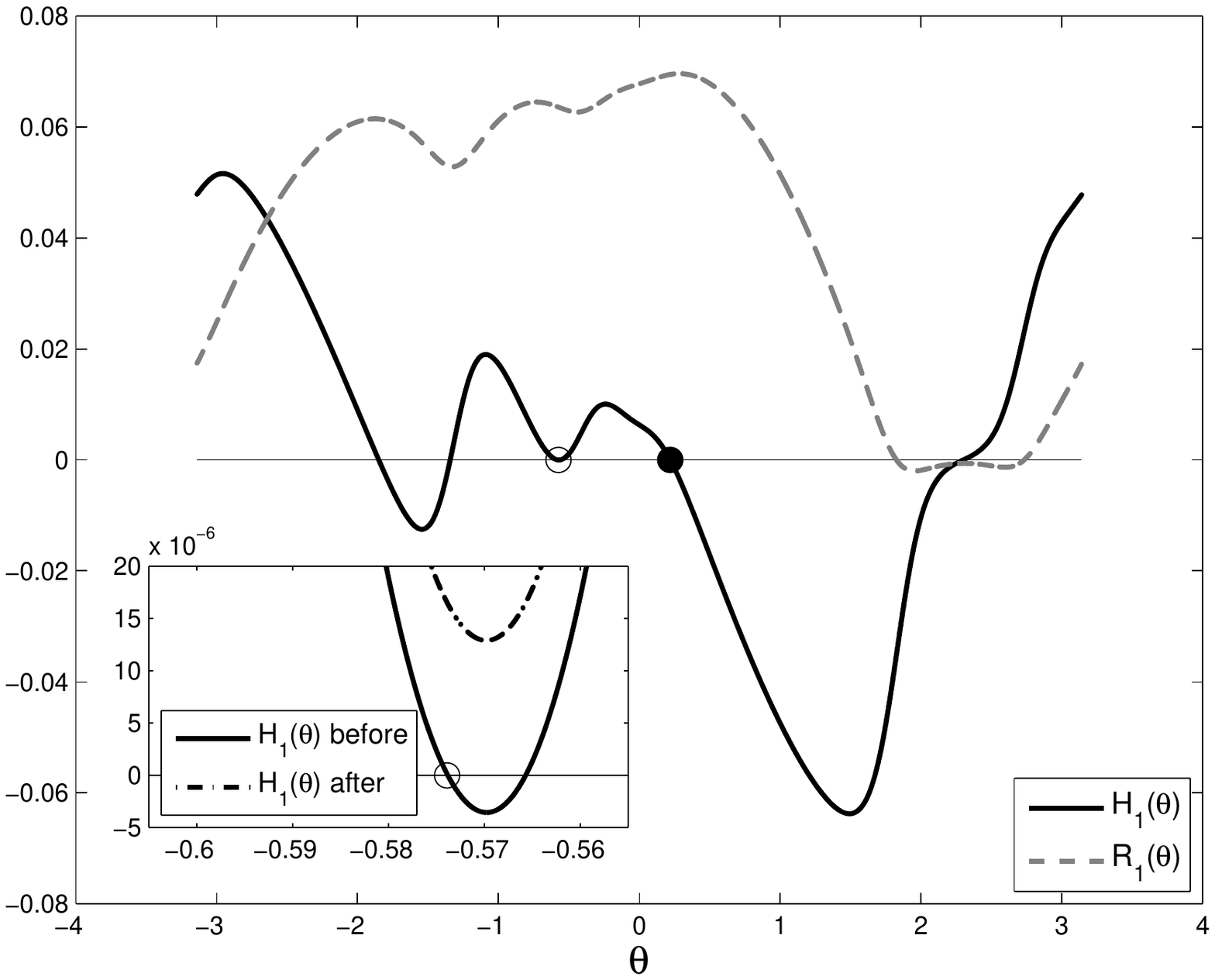}\\
\caption{Functions $H_1$ and $R_1$ (see \eqref{eqn:FR}) at breakdown $t^\ast=1.41$ when root loss (jump of Type I) occurs for particle $i=1$. The dynamics of \eqref{eqn:fo-bz} cannot continue in the direction of the current root $\theta_1 \approx -0.57$ (indicated by an empty circle), as this value would no longer be a root immediately after $t^\ast$. The zoomed-in insert illustrates this scenario: before jump (solid line), when the root indicated by an empty circle is still present, and an extension (dash-dotted line) in the direction of the current velocity, which leads to root loss. A jump in velocity has to be enforced in order to extend the dynamics of \eqref{eqn:fo-bz} beyond breakdown. The new value is indicated by a filled circle --- see Section \ref{sect:jump-sel}.}%
\label{fig:HRjump1411}%
\end{figure}

Assume for now that we have a criteria for setting a velocity jump at $t^\ast = 1.41$, that we reinitialize \eqref{eqn:fo-bz} at the current spatial configuration $\mathbf{x}(t^\ast)$, but in the direction of the new velocity, and that we can continue the time evolution of  \eqref{eqn:fo-bz} until a new breakdown occurs. Anticipating the results, suppose that $\theta_1$ takes after the jump the new value indicated by the filled circle in Figure \ref{fig:HRjump1411} ($\theta_1 \approx 0.22$) and that the evolution of \eqref{eqn:fo-bz} continues in this new direction. The motivating Figure \ref{fig:trajectories full and zoom}(a) from the Introduction corresponds in fact to the same run of \eqref{eqn:fo-bz} as that considered here, and shows this extended trajectory. More precisely, inspect the trajectory $x_1(t)$ of the top left particle (particle 1) indicated by a solid line in Figure \ref{fig:trajectories full and zoom}(a).  The first segment of this trajectory (up to the first breakdown time $t^\ast$ indicated by a square) is the same as the solid grey line in Figure \ref{fig:instab}(c). At $t^\ast=1.41$ the trajectory $x_1(t)$ makes a sharp turn ($\theta_1$ jumps from the empty-circle to the filled-circle value) and then continues until a second breakdown is encountered. This next breakdown,  indicated by the second square along the trajectory of particle 1, is also a breakdown of type I, and can be discussed using similar considerations as for the first jump.  We do not treat this breakdown in detail, but enforce a jump (as discussed in Section \ref{sect:jump-sel}), and continue the evolution.

\begin{figure}[tb]%
\centering
\includegraphics[width=0.6\textwidth]{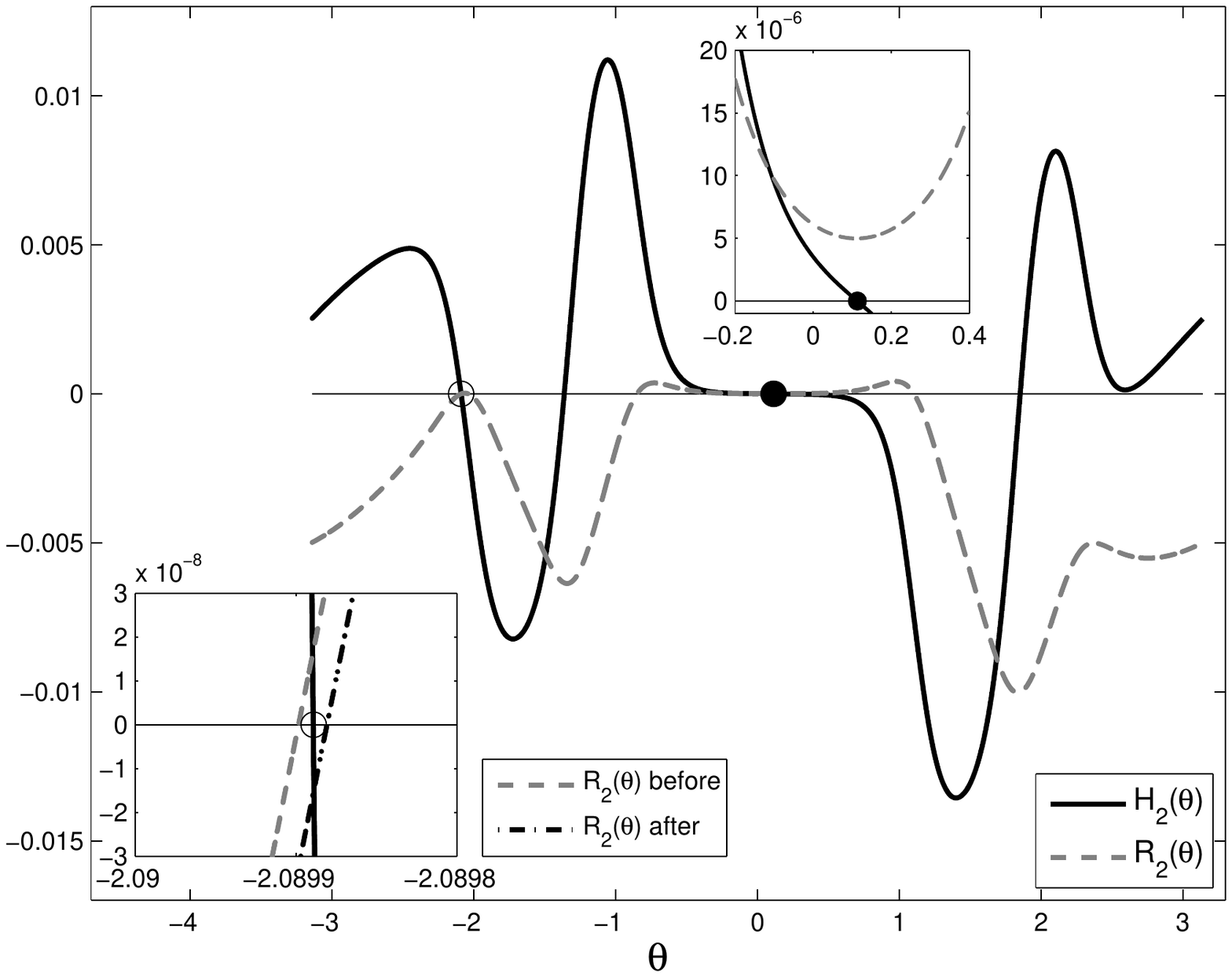}\\
\caption{Functions $H_2$ and $R_2$ (see \eqref{eqn:FR}) at breakdown $t^\ast=34.34$ when particle $i=2$ stops (jump of Type II). Particle $2$ has stopped in the direction $\theta_2^\ast \approx -2.09$ (empty circle), which is simultaneously a root of $H_2$ and $R_2$ (see \eqref{eqn:HRroots}). The bottom left insert shows zoomed-in plots of $H_2$ (solid black) and $R_2$ (dashed grey) near $\theta_2^\ast$, shortly before the breakdown at $t=t^*$. Had the numerical integration continued in the current direction, $R_2(\theta)$ would become negative (dash-dotted black), and the root would no longer be admissible. A jump in velocity has to be enforced in order to extend the dynamics of \eqref{eqn:fo-bz} beyond the stopping breakdown. The post-jump direction is indicated by a filled circle --- see Section \ref{sect:jump-sel}.  The top right insert shows that indeed, this new root is admissible, as $R_2$ is positive (but small) there.}%
\label{fig:HRjump34337}%
\end{figure}

We focus instead on the breakdown indicated by the square on the trajectory of the particle that starts from top right in Figure \ref{fig:trajectories full and zoom}(a); we label this particle as particle $2$.  This  breakdown is of type II. Particle $2$ brakes continuously, as described in Section \ref{sect:modes}, and stops.  Figure \ref{fig:HRjump34337} shows the plots of $H_2$ and $R_2$ at this stopping time $t^\ast = 34.34$. The particle stops in the direction $\theta_2^\ast \approx -2.09$ indicated by the empty circle, which is simultaneously a root of $H_2$ and $R_2$. This is equivalent to the fixed point equation \eqref{eqn fixed point polar coordinates} (for $i=2$) to have the solution $r_2=0$,  $\theta_2=\theta_2^\ast$ (see also \eqref{eqn:HRroots}). In the bottom left insert of Figure \ref{fig:HRjump34337} we show the functions $H_2$ (solid black) and $R_2$ (dashed grey) shortly before breakdown. Since $R_2$ is positive (but very small, note the scale of the vertical axis in the insert) at the root of $H_2$ indicated by the empty circle, the corresponding velocity is admissible. However, by evolving \textit{numerically} \eqref{eqn:fo-bz} in the direction of the current root $\theta_2^\ast \approx -2.09$, $R_2(\theta)$ becomes negative (dash-dotted black line) and the root is no longer admissible. We conclude that beyond stopping time, phase-space trajectories cannot be extended continuously, and a jump in $v_2$ has to occur. We remark that the two graphs of $H_2$ (before breakdown and after extension) nearly coincide and the difference is not visible in the plot. The filled circle in Figure \ref{fig:HRjump34337} indicates the value of $\theta_2$ after the jump (see Section \ref{sect:jump-sel}). We include the top right insert in Figure \ref{fig:HRjump34337} to clarify that $R_2$ is indeed positive at the new $\theta_2$.


\subsection{Jump selection through the relaxation model}
\label{sect:jump-sel}

A central issue in this article is how to continue the solutions of \eqref{eqn:fo-bz} beyond a breakdown time, by enforcing a jump in velocity. Note that, having reached a breakdown time, there could be multiple options for a jump in velocity. For instance, at the breakdown time $t^*=1.41$ in Figure \ref{fig:HRjump1411}, there are three simple roots of $H_1$ which are admissible (that is, $R_1>0$ at these roots). Enforcing a jump in $\theta_1$ to any of these {\em isolated} roots would enable us to continue the dynamics of \eqref{eqn:fo-bz} beyond the breakdown.

The question is how to select which jump to perform. This is done using the relaxation system \eqref{eqn:so}. Based on the   interpretation of this model as including small but positive inertia or response time, we expect that {\em physically relevant} solutions of the anisotropic model \eqref{eqn:fo-bz} should be attained as limits $\ep \to 0$ of solutions to   \eqref{eqn:so}. It would thus be  meaningful to choose the jump that the $\ep$-system selects in the $\ep \to 0$ limit.

We perform runs of the relaxation model \eqref{eqn:so} using three values of $\ep$: $\ep=10^{-2},10^{-3}$, and $10^{-4}$. We initialize \eqref{eqn:so} with a phase-space configuration that corresponds to the numerical run presented in Section  \ref{sec:pos stab d=2} and used in the considerations above:  initial spatial configuration as in Figure \ref{fig:instab}(a), and initial velocity as the fixed point of \eqref{eqn:vi-bz} that corresponds to the stable root $\theta_1\approx -1.00$ in Figure \ref{fig:instab}(b). As discussed and illustrated in Section \ref{sect:conv}, starting from a stable root, we have convergence of the $\ep$-model to solutions of \eqref{eqn:fo-bz}, before a breakdown of \eqref{eqn:fo-bz} occurs. Figure \ref{fig:trajectories full and zoom}(a) shows the trajectories of \eqref{eqn:so}, though on the scale of the figure they are indistinguishable from the solution of \eqref{eqn:fo-bz}.

Upon approaching the first breakdown time of \eqref{eqn:fo-bz}, $t^\ast=1.41$, solutions of \eqref{eqn:so} steepen and approach, via a fast dynamics, a different isolated stable root of \eqref{eqn:fo-bz}. The zoomed plots in Figure \ref{fig:trajectories full and zoom}(b), as well as those in Figure \ref{fig:jump1}, show this evolution of the $\ep$-system near $t^\ast= 1.41$. Figure \ref{fig:trajectories full and zoom}(b) shows the trajectory $x_1^\ep(t)$, while Figures \ref{fig:jump1}(a) and (b) plot $\theta_1^\ep(t)$ and $|v_1^\ep(t)|$, respectively. In each such figure, the fast transition of solutions within an $O(\ep)$ time interval can be observed.  Returning to Figure  \ref{fig:HRjump1411}, and inspecting Figure \ref{fig:jump1}(a), we notice that indeed, the stable root $\theta_1 \approx 0.22$ (filled circle) of the degenerate system is being selected by the $\ep$-model. Note again that this is not the only admissible stable root (with $H_1'<0$, $R_1>0$) available at the jump (see Figure  \ref{fig:HRjump1411}). But in light of the convergence result in Section \ref{subsect:conv}, the selection of $\theta_1$ at the filled circle was in fact expected, and the reason is discussed in the following paragraph.

Consider the adjoined system associated to the $\ep$-model --- see \eqref{eqn:adj-i} and \eqref{eqn:polars} for $i=1$. At a fixed spatial configuration $\mathbf{x}^\ast$, evolving the {\em fictitious} time $\tau \to \infty$ yields indication on the asymptotic stability of a root. Hence, consider hypothetically the adjoined system \eqref{eqn:polars} with $i=1$ for a spatial configuration $\mathbf{x}^\ast_+$ consisting of an infinitesimal extension from $t=t^\ast$ to $t=t^\ast_+$ of the spatial configuration $\mathbf{x}^\ast$ at the jump of the degenerate system \eqref{eqn:fo-bz}, extension taken in the current direction of motion of \eqref{eqn:fo-bz}. The plots of $H_1$ and $R_1$ corresponding to such an extension to $t^\ast_+$ would be infinitesimal perturbations of the plots in Figure \ref{fig:HRjump1411}, where most importantly, the double root indicated by an empty circle is no longer a root of $H_1$ at $t^\ast_+$ (this ``root loss" is the reason for the breakdown, cf.~the insert in Figure \ref{fig:HRjump1411}). Evolving the adjoined system \eqref{eqn:polars} with $i=1$ at the frozen, hypothetical, {\em post-jump} configuration $\mathbf{x}^\ast_+$ is expected to provide the new asymptotically stable root that the $\ep$-system would converge to. The evolution of $\theta_1$ is simply driven by the sign of the right-hand-side in \eqref{eqn:polars-theta} ($i=1$), and since $r_1>0$, this sign is given by $H_1$ at $t^\ast_+$. It is now clear from Figure \ref{fig:HRjump1411} that initializing \eqref{eqn:polars-theta} ($i=1$) with $\theta_1$ near the empty circle, which is the value it had before jump, would result in selecting the stable fixed point indicated by the filled circle. This observation serves as the starting point in designing an efficient numerical method to simulate model \eqref{eqn:fo-bz} (Section \ref{sect:algorithm}).

Model \eqref{eqn:fo-bz} encounters a breakdown of type II at $t^\ast=34.34$, when particle $2$ stops in the current direction $\theta^\ast \approx -2.09$ (the root of $H_2$ indicated by empty circle in Figure  \ref{fig:HRjump34337}). On the contrary, solutions of the $\ep$-model \eqref{eqn:so} continue through $t^\ast$ and capture again a certain jump in direction. Figure \ref{fig:trajectories full and zoom}(c) plots the trajectory $x_2^\ep(t)$ near the second jump $t^\ast=34.34$, while Figures \ref{fig:jump3}(a) and (b) show $\theta_2^\ep(t)$ and $|v_2^\ep(t)|$, respectively. Note indeed that Figure \ref{fig:jump3}(b) captures the braking of particle $2$ that occurs in the degenerate system ($|v_2|$ reaches order $O(10^{-8})$). The difference though is that solutions of the $\ep$-system do not actually stop, as particle $2$ changes direction (see Figure \ref{fig:jump3}(a) where $\theta_2$ evolves fast from $\approx -2.09$ to $\approx 0.11$), picks up a higher velocity (of order $O(10^{-5})$), and continues the motion. This fast transition results in a very sharp turn in the trajectory,  as illustrated in the zoomed plot Figure \ref{fig:trajectories full and zoom}(c) (see also the insert in the figure).

By inspecting Figure \ref{fig:HRjump34337} one observes that the $\ep$-system has selected the jump to root $\theta_2 \approx 0.11$ indicated by the filled circle. In this case this was in fact the only admissible root of $H_2$ at $t^\ast$, as the others have $R_2<0$. However, were there more admissible roots, it is not as clear as it was for the type I jump in Figure \ref{fig:HRjump1411}, whether similar considerations regarding the adjoint system \eqref{eqn:polars} can be used to predict the selection of the post-jump velocity. First, there is no natural extension (from $t^\ast$ to $t^\ast_+$) of a configuration $\mathbf{x}^\ast$ at a breakdown that involves a {\em resting} particle. In a numerical simulation however,  this point is less relevant, as the numerical value of a particle that attempts to stop gets very small, but it doesn't actually reach zero. Hence, extending the numerics by a small amount into a post-jump configuration is possible (this was done for instance to produce the insert in Figure \ref{fig:HRjump34337}). Second, from a theoretical point of view, the evolution $\tau \to \infty$ in \eqref{eqn:polars-theta} with $i=2$, at an infinitesimally extended spatial configuration $\mathbf{x}^\ast_+$, cannot be argued as for jump I, by invoking the sign of the right-hand-side (in this case, the sign of $H_2$ at $t^\ast_+$). The full two-dimensional evolution of the adjoint system \eqref{eqn:adj-i}  would have to be employed instead, and issues such as the domain of influence and getting attracted into a certain fixed point, are more subtle. We conclude by noting that in practice, for numerical simulations, the frozen/adjoint-system idea seems to work fine for jumps of type II as well, it is just its theoretical foundation that is less solid than for jumps I. Alternatively, one could use the {\em real} time evolution of the $\ep$-system near the breakdown in order to select a jump (as discussed below in Section \ref{sect:algorithm}).

The numerical observations reported in this section have been confirmed with various other simulations, involving different initial conditions and larger number of particles. The two types of jumps discussed here and the shock-capturing of the $\ep$-system are typical findings.


\begin{figure}%
\centering
\includegraphics[totalheight=0.24\textheight]{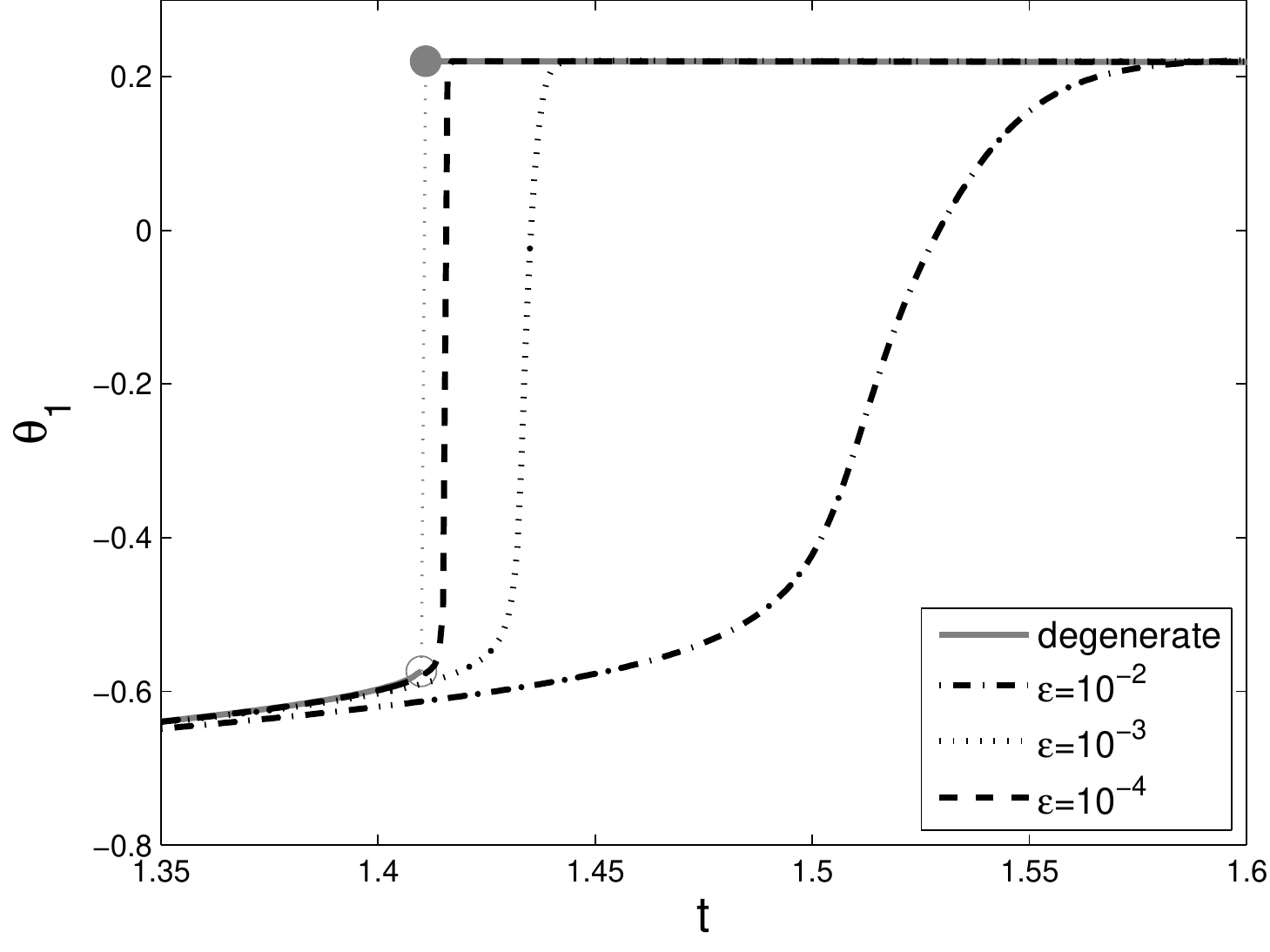}
$~~~$
\includegraphics[totalheight=0.25\textheight]{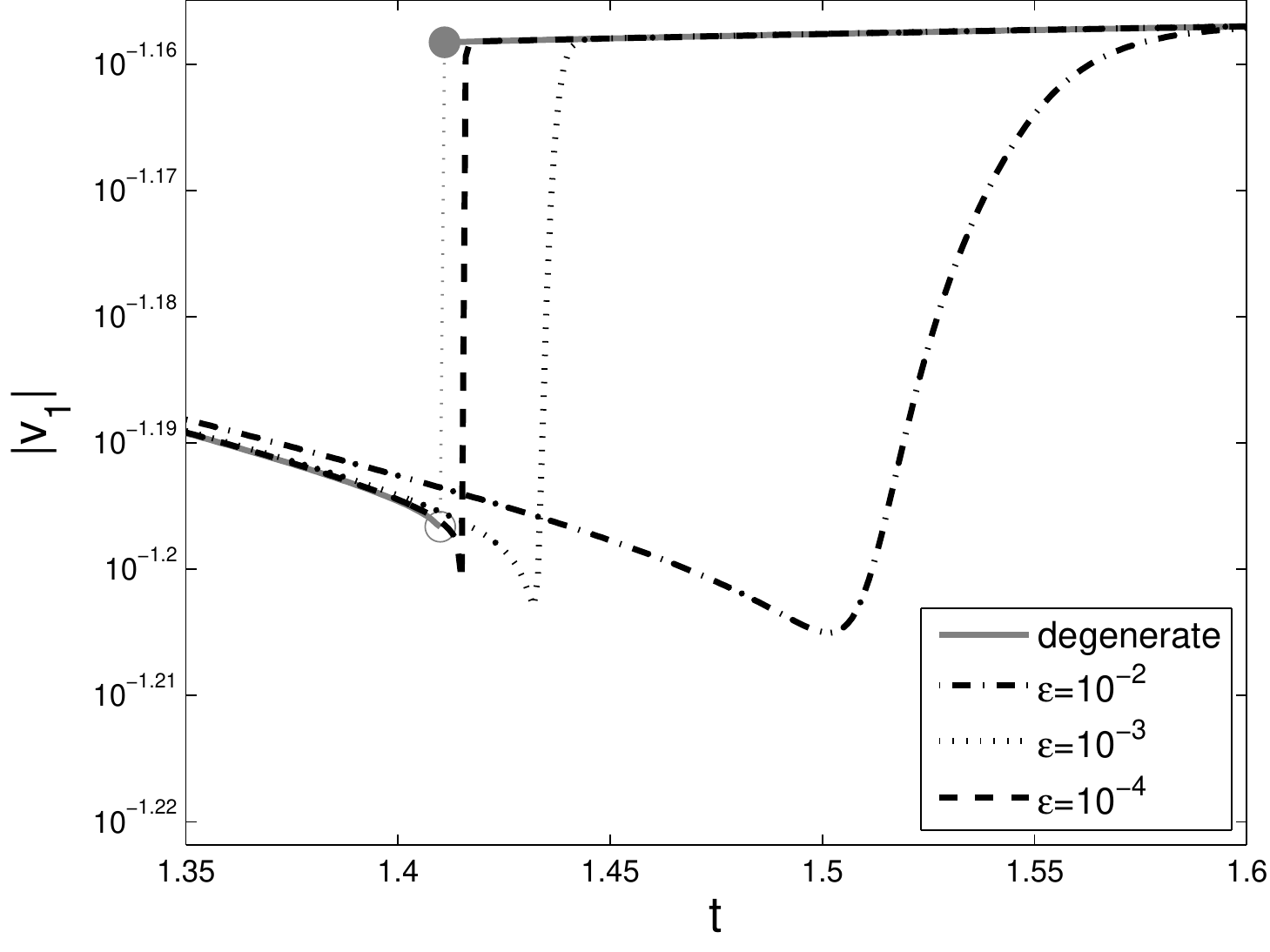}\\
 (a) \hspace{7.2cm} (b)
\caption{Solution of the relaxation system \eqref{eqn:so} near the first breakdown time $t^\ast = 1.41$ of the anisotropic model \eqref{eqn:fo-bz}, due to root loss. The plots show (a) $\theta_1^\ep(t)$ and (b) $|v_1^\ep(t)|$. Three values of $\ep$ are used to illustrate shock capturing: $10^{-2}, 10^{-3}$, and $10^{-4}$.  Near $t^\ast$, the direction changes from $\theta_1 \approx -0.57$ to  $\theta_1 \approx 0.22$, values indicated, respectively, by the empty and filled circles in Figure \ref{fig:HRjump1411}. Complete trajectories for this run can be found in Figure \ref{fig:trajectories full and zoom}(a). The breakdown time $t^\ast=1.41$ is indicated by the first square along the trajectory of the top-left particle (particle 1). A zoomed trajectory $x_1^\ep(t)$ near $t^\ast$ can be found in Figure \ref{fig:trajectories full and zoom}(b).}%
\label{fig:jump1}%
\end{figure}

\begin{figure}%
\centering
\includegraphics[totalheight=0.24\textheight]{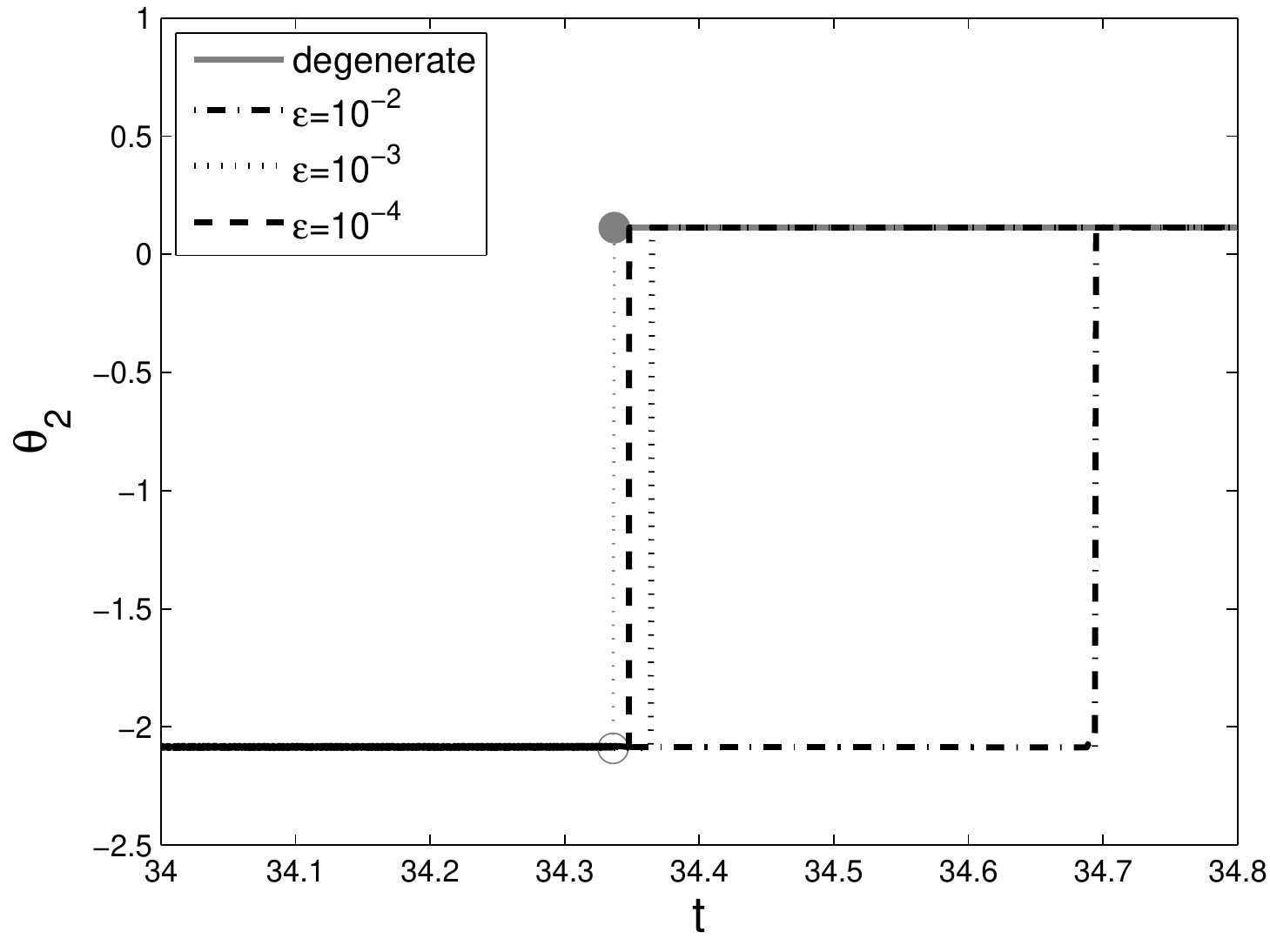}
$~~~$
\includegraphics[totalheight=0.25\textheight]{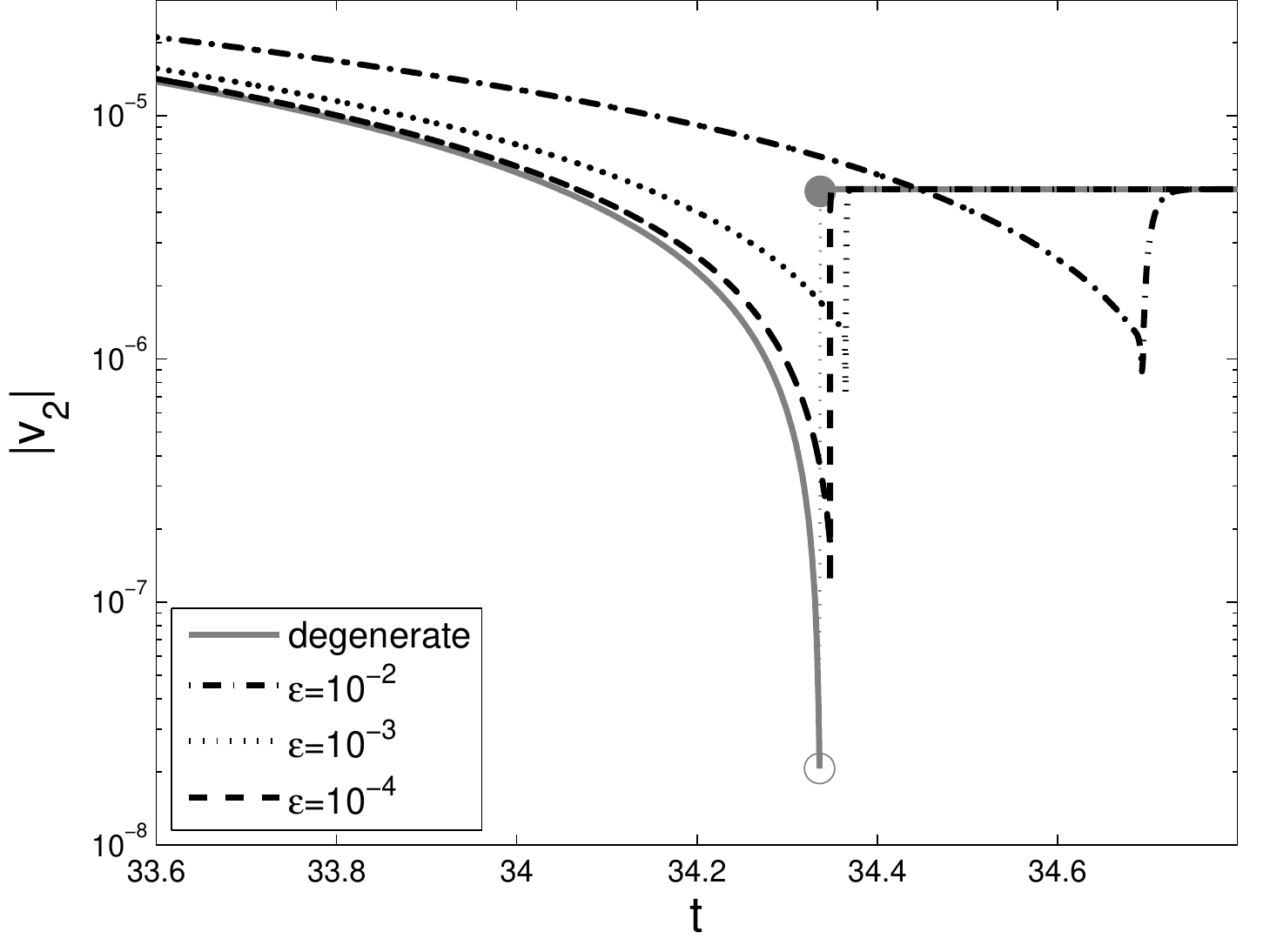}\\
 (a) \hspace{7.2cm} (b)
\caption{Solution of the relaxation system \eqref{eqn:so} near the third breakdown time $t^\ast = 34.34 $ of  \eqref{eqn:fo-bz}, due to stopping of particle 2. The plots show (a) $\theta_2^\ep(t)$ and (b) $|v_2^\ep(t)|$. Three values of $\ep$ are used to illustrate shock capturing: $10^{-2}, 10^{-3}$, and $10^{-4}$. The direction changes from $\theta_2 \approx -2.09$ to  $\theta_2 \approx 0.11$, values indicated, respectively, by the empty and filled circles in Figure \ref{fig:HRjump34337}. The complete trajectories can be found in Figure \ref{fig:trajectories full and zoom}(a), where the breakdown time $t^\ast=34.34$ is indicated by the square along the trajectory of particle 2 (top right). A zoomed trajectory $x_2^\ep(t)$ near breakdown can be found in Figure \ref{fig:trajectories full and zoom}(c).}%
\label{fig:jump3}%
\end{figure}


\section{Long-time evolution and concluding remarks}
\label{sect:algorithm}

\paragraph{Numerical implementation of \eqref{eqn:fo-bz} in two dimensions.}

Evolving the relaxation system \eqref{eqn:so} with small $\ep$ for large times is not practically feasible. The numerical strategy for the long-time evolution of \eqref{eqn:fo-bz} is to run the anisotropic model through its intervals of continuity and use the $\ep$-model only to capture the jumps.  For jumps of type I, this procedure is rather easy to implement in two dimensions, as illustrated in Section \ref{sect:jump-sel}. Indeed, suppose that in a numerical simulation of \eqref{eqn:fo-bz} a root loss has been identified in the discrete time step from $t_{n-1}$ to $t_{n}$. That is, for some particle $i$, the numerical velocity $v_i^{n}$ at time $t_n$ is no longer a fixed point of $\mathcal{F}_i$. Then, by ``freezing" the {\em post-jump} spatial configuration~$\mathbf{x}^{n}$ at the time~$t_n$, one can run the adjoint system \eqref{eqn:adj-i} with the fictitious time $\tau \to \infty$, in order to select the new, asymptotically stable root. Rename this root $v_i^n$ and then continue the evolution of~\eqref{eqn:fo-bz}. This procedure is the time-discrete version of the considerations from Section \ref{sect:jump-sel} on the selection of a jump by an infinitesimal extension of the spatial configuration at the breakdown time.

Jumps of type II can be similarly recovered, by freezing the post-jump spatial configuration. As explained in Section \ref{sect:jump-sel} above, this procedure is less theoretically grounded for jumps of type II, but we found that it works well in practice and captures the correct jumps.  We confirmed this with full, {\em real}-time evolutions of the relaxation model through type II discontinuities. That is, after detecting a jump in the discrete time interval from $t_{n-1}$ to $t_{n}$, return to the {\em pre-jump} phase-space configuration $(\mathbf{x}^{n-1}, \mathbf{v}^{n-1})$ at time $t_{n-1}$, initialize \eqref{eqn:so} with this data, and run the relaxation system with a fine time resolution to capture the steep solution that selects the post-jump root $\mathbf{v}^{n}$.

\paragraph{Long-time behaviour.}
An extensive numerical study of the long-time behaviour of solutions to \eqref{eqn:fo-bz} is beyond the scope of the present paper.
We only report briefly our observations.
The main feature is that the dynamics slows down significantly after a relatively short initial interval. For instance, particles in the numerical simulation considered above (referred to as IC 1 here) reach velocities of order $O(10^{-4})$ by the stopping breakdown time $t^\ast = 34.34$, and continue to decrease steadily after the jump. In Figure \ref{fig:long-timeN4}(b) we plot the maximum speed $\operatorname{max}_i |v_i|$ over time, to $t=5,000$.  We also considered the long time run corresponding to the same initial spatial configuration from Figure \ref{fig:instab}(a), but with an initial velocity $v_1$ pointing in the other stable direction, $\theta_1 \approx -1.78$; we refer to this initial condition as IC 2. The evolution of $\operatorname{max}_i |v_i|$ is also shown in Figure \ref{fig:long-timeN4}(b), with similar qualitative behaviour as for~IC 1.

The full evolution of the trajectories for IC 2 is shown in Figure \ref{fig:long-timeN4}(a), with the final configuration at $t=5,000$ indicated by filled diamonds. The empty diamonds in the figure represent the state at $t=5,000$ of the run with IC 1. We do not plot the full evolution of the trajectories corresponding to IC 1, since at the scale of the figure these would be indistinguishable from the solutions shown in Figure \ref{fig:trajectories full and zoom}(a). Note that the two sets of configurations have different centres of mass.  The centre of mass is not being conserved by the anisotropic model \eqref{eqn:fo-bz}, as it is for the isotropic model \eqref{eqn:fo-nobz}. The two configurations are close in shape to a rhombus, suggesting non-symmetrical states such as ellipses as possible quasi-equilibria.

Figure \ref{fig:long-timeN4} shows that both runs feature a fast initial dynamics (involving several jumps of both types), followed by slow motion.  The numerical results suggest that velocities continue to decrease indefinitely, and the system reaches a quasi-steady state. Stopping jumps become more typical at low speeds, as particles make small jiggles, turning toward and from the others, trying to reach an equilibrium. This jiggling aspect is not present in the isotropic model, as there, the unobstructed sensing of the others drives the particles quickly into an equilibrium configuration.


\begin{figure}%
\centering
\includegraphics[totalheight=0.24\textheight]{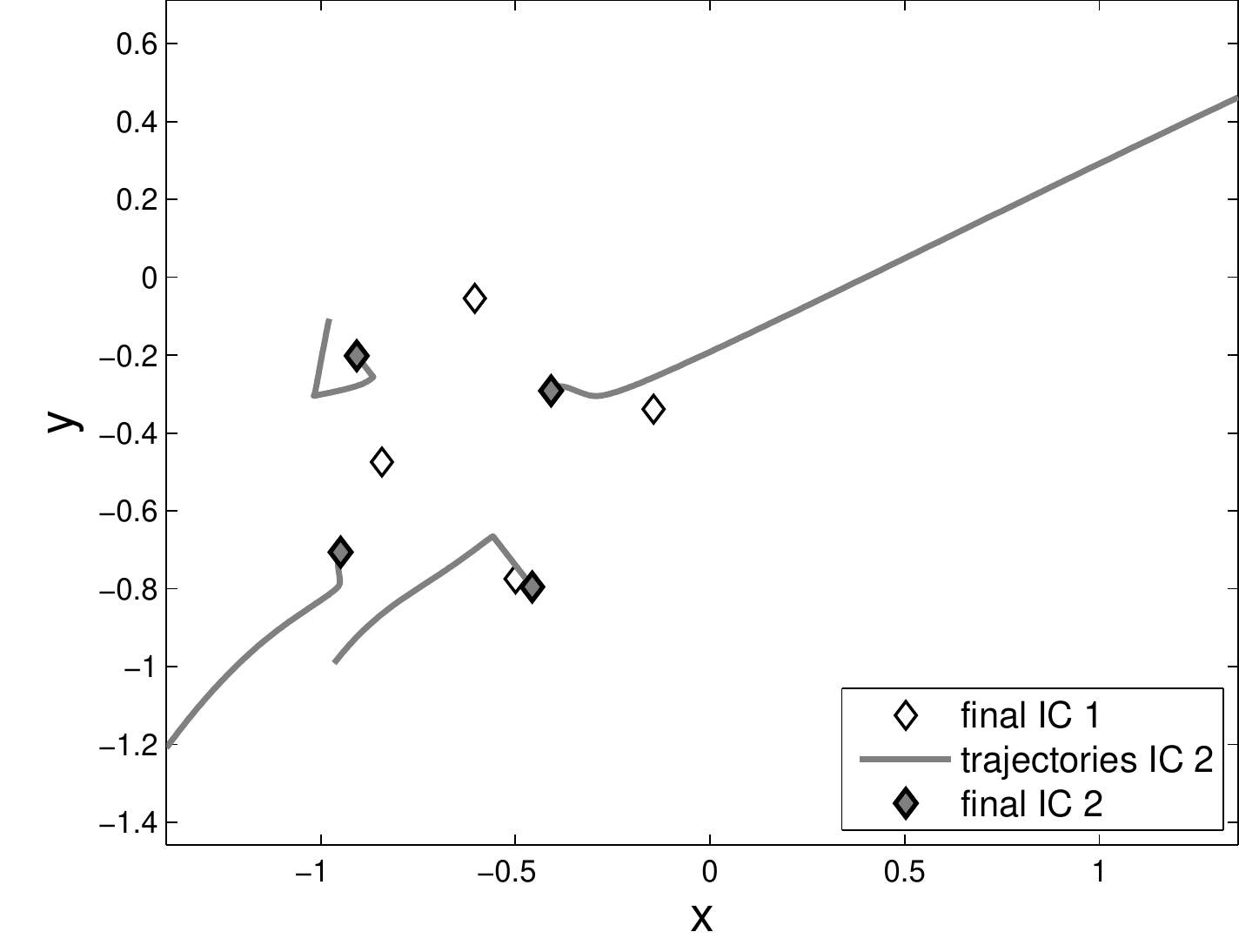}
$~~~$
\includegraphics[totalheight=0.25\textheight]{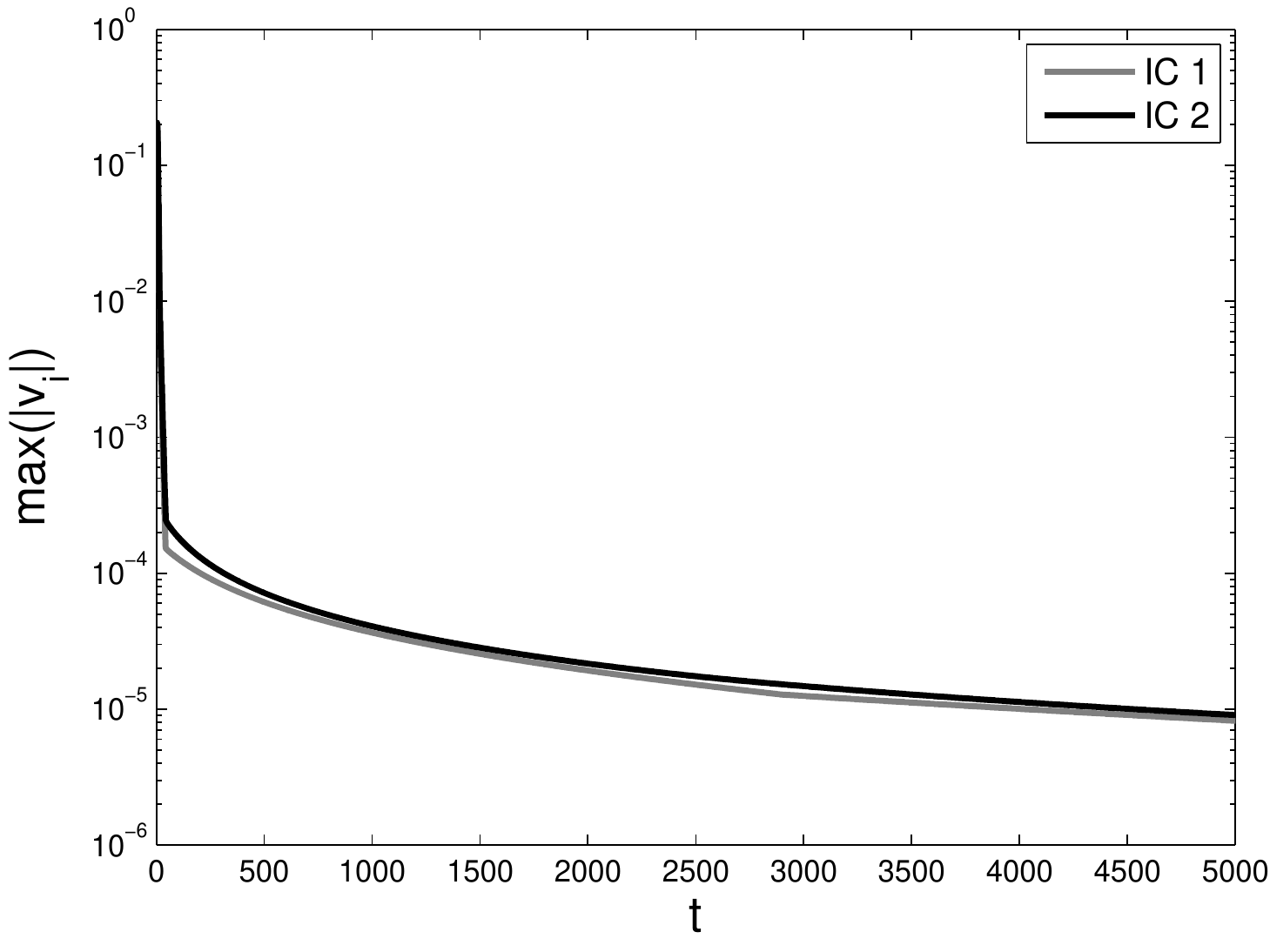}\\
 (a) \hspace{7.2cm} (b)
\caption{Long-time behaviour of \eqref{eqn:fo-bz}. (a) The solid grey lines represent the evolution of \eqref{eqn:fo-bz} starting from the initial configuration in Figure \ref{fig:instab}(a), in the direction of the stable root $\theta_1\approx -1.78$ --- we refer to this initial condition as IC 2. The initialization in the direction of the other stable root, $\theta_1\approx -1.00$, is referred to as IC 1. The filled and empty diamonds indicate the configurations of the two runs at $t=5,000$.  (b) Plot of $\operatorname{max}_i |v_i|$ over time, for the two sets of initial conditions. Particles reduce their speed and seem to approach a quasi-equilibrium.}%
\label{fig:long-timeN4}%
\end{figure}

\paragraph{Concluding remarks.} We showed in this paper that accounting for  anisotropy in the aggregation model \eqref{eqn:fo-nobz} brings up new  interesting issues, both analytically and numerically. In particular, we reinstated the role of the relaxation model \eqref{eqn:so}, which was initially used to formally derive the first-order model \eqref{eqn:fo-nobz}, but then mostly ignored by researchers on this topic. We end by noting  that, as the number of particles becomes large, accounting for all jumps that take place in the dynamics of \eqref{eqn:fo-bz} becomes quite challenging. The natural resort for the $N$ large case is the anisotropic extension of the continuum model \eqref{eq:aggre}, which is equation \eqref{eqn:rhot} with $v$ given implicitly by:
\[
v(x) = \int \nabla K(|x-y|) g \left( \frac{x-y}{|x-y|} \cdot \frac{v(x)}{|v(x)|}\right) \rho(y) dy.
\]
Here $K$ and $g$ have the same meaning as throughout the paper. Investigating such an anisotropic extension of the continuum model is an important, and quite challenging, new research direction.


\section*{Acknowledgments}
The authors thank Adrian Muntean for various thoughtful suggestions during the work on this paper. JE thanks Giovanni Bonaschi, Manh Hong Duong, Patrick van Meurs, Georg Prokert and in particular Mark Peletier for their input that led to the generalized concept of a fixed point. JE acknowledges the financial support received from the Netherlands Organisation for Scientific Research (NWO), Graduate Programme 2010.

\bibliographystyle{plain}
\bibliography{lit}


\end{document}